\documentclass[12pt]{amsart}
\usepackage{amsmath,amstext,amssymb,amsopn,amsthm}
\usepackage{wasysym}
\usepackage[top=20mm,bottom=15mm,left=20mm,right=20mm]{geometry}
\usepackage{graphicx}

\usepackage{hyperref}
\usepackage{mathrsfs}

\usepackage{amsmath, amsfonts, amssymb, amsthm ,mathrsfs, bbm,dsfont}
\usepackage{color}

\newtheorem{theorem}{Theorem}[section]

\newtheorem{corollary}[theorem]{Corollary}
\newtheorem{lemma}[theorem]{Lemma}

\newtheorem{proposition}[theorem]{Proposition}

\theoremstyle{remark}

\usepackage{delarray}
\usepackage{enumerate}

\usepackage{color}
\definecolor{dp}{RGB}{255,0,0}
\definecolor{kb}{RGB}{0,255,0}
\definecolor{tj}{RGB}{255,0,255}

\numberwithin{equation}{section}
\numberwithin{theorem}{section}
\numberwithin{figure}{section}

\newcommand{\rf}[1]{\eqref{#1}}

\newcommand{\bbfR}{{\mathbb R}}

\newcommand{\bbfN}{{\mathbb N}}

\newcommand{\ve}{{\varepsilon }}

\newcommand{\R}{{\mathbb R}}
\date{\rule{0pt}{15pt}\today}
\newcommand{\Rd}{{{\mathbb R}^d}}
\newcommand{\RR}{\mathbb{R}}
\newcommand{\NN}{\mathbb{N}}
\newcommand{\tEe}{{\tilde{\mathcal E}}}
\DeclareMathOperator{\supp}{supp}

\def \eps{\varepsilon}

\def \tp{\tilde{p}}

\title{Fractional Laplacian with Hardy potential}
\author[K. Bogdan]{Krzysztof Bogdan}
\address{Faculty of Pure and Applied Mathematics,
Wroc\l aw University of Science and Technology,
Wyb. Wyspia\'nskiego 27, 50-370 Wroc\l aw, Poland}
\email{krzysztof.bogdan@pwr.edu.pl}
\author[T. Grzywny]{Tomasz Grzywny}
\email{tomasz.grzywny@pwr.edu.pl}
\author[T. Jakubowski]{Tomasz Jakubowski}
\email{tomasz.jakubowski@pwr.edu.pl}
\author[D. Pilarczyk]{Dominika Pilarczyk}
\email{dominika.pilarczyk@pwr.edu.pl}
\thanks{The first and the second authors were partially supported by the NCN grant 2014/14/M/ST1/00600. The third author was supported by the NCN grant 2015/18/E/ST1/00239}

\subjclass[2010]{Primary 47D08, 60J35; Secondary 31C05, 46E35}

\keywords{fractional Laplacian, Hardy inequality, heat kernel}
\begin{document} 

\maketitle

\begin{abstract}
We give sharp two-sided estimates of the semigroup generated by the fractional Laplacian plus the Hardy potential on $\Rd$, including the case of the critical constant. We use Davies' method  back-to-back with a new method of integral analysis of Duhamel's formula.
\end{abstract}

\section{Main result and Introduction}
Let $d\in \bbfN:=\{1,2,\ldots\}$, $\alpha\in (0,2)$ and $\alpha<d$. We 
consider the following Schr\"odinger operator 
on $\Rd$ 
with the Hardy  potential,
   \begin{equation}
   \label{eq:SchrOp}
   \Delta^{\alpha/2}+\kappa|x|^{-\alpha}.
   \end{equation} 
Here $\Delta^{\alpha/2}$ is the fractional Laplacian  and 
 $\kappa$ is a positive number.
Let
\begin{equation*}
   \kappa^*=\frac{2^{\alpha} \Gamma((d+\alpha)/4)^2 }{\Gamma((d-\alpha)/4)^{2}}.
\end{equation*}
This is the best constant in the Hardy inequality for the quadratic form of $\Delta^{\alpha/2}$:
\begin{equation}\label{eq:Hi}
\mathcal E [f]\ge \int_\Rd f(x)^2\kappa^* |x|^{-\alpha}dx,\qquad f\in L^2(\Rd),
\end{equation}
see below for definitions;  see Herbst \cite[Theorem~2.5]{MR0436854}, Frank and Seiringer \cite[Theorem~1.1]{MR2469027} and Bogdan, Dyda and Kim \cite[Proposition~5]{MR3460023} for the result. 
If $0<\kappa\le \kappa^*$, then there is  a unique number $\delta$ such that 
$$
0<\delta\leq (d-\alpha)/2 \quad \text{ and } \quad
\kappa=\frac{2^\alpha\Gamma(\frac{\delta+\alpha}{2}) \Gamma(\frac{d-\delta}{2})}{\Gamma(\frac{\delta}{2})\Gamma(\frac{d-\delta-\alpha}{2})},
$$
see Figure~\ref{fig:k}. Our main theorem is as follows.
\begin{theorem}\label{thm:main}
Let $0<\alpha<2\wedge d$. If $0<\kappa\leq \kappa^*$, then the Schr\"odinger operator \eqref{eq:SchrOp} has the heat kernel $\tilde{p}$ which is jointly continuous on $(0,\infty)\times (\Rd\setminus\{0\})^2$ and satisfies
\begin{equation*}
\label{eq:mainThmEst}
\tilde{p}(t,x,y)\approx \left(1+t^{\delta/\alpha}|x|^{-\delta}  \right)\left(1+t^{\delta/\alpha}|y|^{-\delta}  \right) \left( t^{-d/\alpha}\wedge \frac{t}{|x-y|^{d+\alpha}} \right),\quad x,y\neq0,\ t>0.
\end{equation*}
\end{theorem}
Here and below
we write $f \approx g$ if $f,g\ge 0$ and $c^{-1} g \le f \le c g$
for some positive number $c$ (comparability
constant). Such estimates are called sharp or two-sided.
We call $\kappa$, and $\Delta^{\alpha/2}+\kappa|x|^{-\alpha}$, subcritical if $0<\kappa<\kappa^*$, critical  if $\kappa=\kappa^*$ and supercritical  if $\kappa>\kappa^*$.

The subject of the paper can be tracked down to Baras and Goldstein \cite{MR742415}, who proved the existence of nontrivial nonnegative solutions of the classical heat equation
$\partial_t=\Delta+\kappa|x|^{-2}$ in $\Rd$ for $0\le \kappa\le (d-2)^2/4$, and nonexistence of such solutions, that is explosion, for bigger constants $\kappa$.
Vazquez and Zuazua \cite{MR1760280} studied the Cauchy problem and spectral properties of the operator in bounded subsets of $\Rd$ using the improved Hardy-Poincar\'e inequality, and they used weighted Hardy-Poincar\'e inequality in the
more delicate case of the whole of $\Rd$.
Sharp upper and lower bounds for the heat kernel of the Schr\"odinger operator $\Delta+\kappa|x|^{-2}$ were obtained by Liskevich and Sobol \cite[p.~365, Example 3.8, 4.4 and 4.10]{MR1953267}  for $0<\kappa< (d-2)^2/4$. 
Milman and Semenov proved the upper and lower
bounds  
for $\kappa\le (d-2)^2/4$, see 
\cite[Theorem 1]{MR2064932}, \cite{MR2180080} and \cite[Section 10.4]{MR2218016}. Note that
Moschini and Tesei \cite[Theorem 3.10]{MR2328115} gave estimates for the subcritical case  in bounded domains, and
Filippas, Moschini and Tertikas \cite{MR2308757} obtained the critical case  in bounded domains. 
Recently Metafune, Sobajima and Spina \cite{Metafune2017} extended the results of Milman and Semenov to sigular  gradient-and-Schr\"odinger perturbations of $\Delta$.
Because of the borderline singularity  of the function
$\Rd\ni x\mapsto\kappa |x|^{-2}$   at the origin, the choice of $\kappa$ influences the growth rate of the heat kernel at the origin. 

The operators $\Delta+\kappa |x|^{-2}$ play distinctive roles in limiting and self-similar phenomena in probability \cite{2017-LM-EP} and partial differential equations \cite{MR3020137}. This results in part from the scaling of the corresponding heat kernel, which is similar to that of the Gauss-Weierstrass kernel. Analogous applications are expected for the Hardy perturbations of the fractional Laplacian, see \cite{2017-DP-hjm} for first attempts.
We also note that the effect of Schr\"odinger perturbations of $\Delta$ is much milder if $\kappa|x|^{-2}$ is replaced by  functions 
in appropriate Kato classes. 
We refer to  Bogdan and Szczypkowski \cite[Section~1, 4]{MR3200161} for references and results on  Gaussian bounds for Schr\"odinger heat kernels along with an approach based on the so-called 4G inequality. The case when even the Gaussian constants do not deteriorate is discussed by Bogdan, Dziuba\'nski and Szczypkowski \cite{2017arXiv170606172B}. 
Estimates of Schr\"odinger perturbations of general operators and their transition semigroups are given by Bogdan, Hansen and Jakubowski in \cite{MR2457489}, with focus on situations with 3G inequality. Bogdan, Jakubowski and Sydor \cite{MR3000465} and  Bogdan, Butko and Szczypkowski \cite{MR3514392} estimate Schr\"odinger perturbations of integral kernels which are not necessarily semigroups.
The results of these paper give comparability or near comparability of the perturbed kernel with the original one. In fact, the explicit estimate in \cite[Theorem~3]{MR3000465} suffices for many applications.

The 
operator \eqref{eq:SchrOp}
cannot be handled by the methods of these papers  -- as we see in Theorem~\ref{thm:main}, the heat kernel of $\Delta^{\alpha/2}+\kappa|x|^{-\alpha}$ has  different growth than the heat kernel of $\Delta^{\alpha/2}$.
It is proved 
by Abdellaoui, Medina, Peral and Primo
 \cite{MR3492734, MR3479207} that for $\kappa>\kappa^*$ the operator has no weak positive supersolutions and the phenomenon of complete and instantaneous blow-up occurs, whereas for $\kappa\le \kappa^*$ nontrivial nonnegative solutions of \eqref{eq:SchrOp} exists and have certain Sobolev regularity.
Concerning the estimates of the heat kernel of $\Delta^{\alpha/2}+\kappa|x|^{-\alpha}$,
the main contributions to date are by Frank, Lieb and Seiringer
\cite[Proposition~5.1]{MR2425175}
and BenAmor \cite[Theorem~3.1]{2016arXiv160601784B}.
 The first paper gives estimates of the heat kernel of  the 
 operator $|x|^{-\beta}\left(\Delta^{\alpha/2}+\kappa|x|^{-\alpha}-1\right)$, with suitable $\beta>0$. These, however, do not translate directly into heat kernel estimates  for
$\Delta^{\alpha/2}+\kappa|x|^{-\alpha}$.
The second paper \cite{2016arXiv160601784B} gives the upper bound for 
the heat kernel of $\Delta+\kappa|x|^{-\alpha}$
with the Dirichlet conditions on bounded open subsets of $\bbfR^d$. This case is 
qualitatively different from the case of the whole of $\Rd$ because available supermedian  functions used for Doob's conditioning in Davies' method are  locally, but not globally bounded from below.
In the classical setting of $\Delta$, the distinction is  known since \cite{MR1760280}, \cite{MR2328115} and \cite{MR2034290}. In fact Milman and Semenow \cite[Theorem~A]{MR2064932} give a general framework for heat kernel estimates, with the lower-boundedness of the nearly supermedian function being an important assumption. 

In the present paper we show that one can use $x\mapsto |x|^{-\delta}+1$ to estimate $\tp$.
To prove Theorem~\ref{thm:main} 
we develop new integral analysis of the perturbation (Duhamel's) formula using this function.
The integral analysis does not use quadratic forms, but the reader may ask a legitimate question, which we asked ourselves,
whether the usual method based on quadratic forms, namely Davies' method, can be employed to the same end.
The outcome of our integral analysis was announced at the conference Probability and Analysis in B\c{e}dlewo, 15-19 May 2017, but we struggled long to apply Davies' method in  the critical case.
We are now able to present both methods back-to-back in the full range of $\kappa$. This adds a new approach to the toolbox of heat kernel estimates and  may also shed  some light on Davies' method, especially  concerning the domain of the  quadratic form of $\Delta^{\alpha/2}+\kappa^*|x|^{-\alpha}$. We note that the integral analysis automatically gives the continuity of the heat kernel and the estimates hold everywhere, rather than $a.e.$. The reader may even be surprised to see by inspection how smooth this works in the critical case.

On a general level we use
Doob-type conditioning both  for the integral analysis of the heat kernel and in the analysis of its quadratic form.
In doing so we rely on the explicit construction of supermedian functions proposed by Bogdan, Dyda and Kim \cite{MR3460023}. In fact, for $\Delta^{\alpha/2}$ we refine the findings of \cite{MR3460023}, to prove that  the function $|x|^{-\delta}$ is invariant (a ground state) for $\tilde{p}$. Then we prove that
$|x|^{-\delta}+1$ is nearly supermedian for $\tp$,
see Corollary~\ref{cor:sh}.
The latter function is bounded from below, which is 
crucial 
for both methods of analysis of $\tp$ presented in this paper. The importance of such functions for 
Davies' method is known at least since Milman and Semenov \cite[Theorem~A]{MR2064932}.

The structure of the paper is as follows. In Section~\ref{sec:prel} we define the heat kernel $p$ of $\Delta^{\alpha/2}$ and its Schr\"odinger perturbation $\tp$ by the Hardy potential $q(x)=\kappa|x|^{-\alpha}$. In Section~\ref{sec:ia} we analyze auxiliary  integrals, e.g. those of the form $\int_{\Rd} \tp(t,x,y)|y|^{-\beta}dy$ and use them to control $\tp$. The main results of this section are Theorem \ref{thm:1} and Proposition  \ref{lem:int_tpinfty} on nearly supermedian functions, summarized in Corollary~\ref{cor:sh}. In Section~\ref{sec:pt1} we prove Theorem \ref{thm:main}. We also reprove the instantaneous blow-up: $\tp\equiv\infty$ in the supercritical case $\kappa>\kappa^*$. In Section~\ref{sec:P} we discuss the quadratic form $\tilde{\mathcal{E}}$ of $\tp$ and  
 apply Davies' method for the upper bound of $\tp$ in Section~\ref{sec:Dm}. 

\section*{Acknowledgements}
We thank Rupert Frank for detailed comments on \cite{MR2425175}. We thank Rodrigo Ba\~nuelos, Bart\l{}omiej Dyda, Jerome Goldstein, Martin Hairer, Panki Kim, Giorgio Metafune and Edwin Perkins for helpful discussions.

\section{Preliminaries}\label{sec:prel}

Our setting is as follows. Let 
$$
\nu(y)
=\frac{ \alpha 2^{\alpha-1}\Gamma\big((d+\alpha)/2\big)}{\pi^{d/2}\Gamma(1-\alpha/2)}|y|^{-d-\alpha}\,,\quad y\in \Rd\,.
$$
The 
coefficient  above is chosen in such a way that 
\begin{equation}
  \label{eq:trf}
  \int_{\Rd} \left[1-\cos(\xi\cdot y)\right]\nu(y)dy=|\xi|^\alpha\,,\quad
  \xi\in \Rd\,.
\end{equation}
The fractional Laplacian for (smooth compactly supported) test functions $\varphi\in C^\infty_c(\Rd)$ is
\begin{equation*}
  \Delta^{\alpha/2}\varphi(x) = 
  \lim_{\varepsilon \downarrow 0}\int_{|y|>\varepsilon}
  \left[\varphi(x+y)-\varphi(x)\right]\nu(y)dy\,,
  \quad
  x\in \Rd\,.
\end{equation*}
In terms of the Fourier transform \cite[Section~1.1.2]{MR2569321},
$\widehat{\Delta^{\alpha/2}\varphi}(\xi)=-|\xi|^{\alpha}\hat{\varphi}(\xi)$.

We use ``:='' to indicate definitions, e.g. $a\land b := \min\{a, b\}$ and $a\vee b := \max\{a, b\}$. 
In statements and proofs we let $c_i$ denote constants whose exact values are unimportant. 
 These are determined anew
in each statement and each proof. We only consider Borel measurable  functions.

\subsection{The semigroup of $\Delta^{\alpha/2}$}
We consider the convolution semigroup of functions
\begin{equation}
  \label{eq:dpt}
  p_t(x):=\frac{1}{(2\pi)^d}\int_ \Rd e^{-t|\xi|^\alpha}e^{-ix\cdot\xi}\,dx\,,\quad\ t>0,\ x\in
  \Rd\,.
\end{equation}
According to 
(\ref{eq:trf}) 
and the L\'evy-Khinchine formula,
each $p_t$ is a probability density function and $\nu(y)dy$ is the L\'evy measure of the semigroup, 
see, e.g., 
\cite{MR2569321}.
From (\ref{eq:dpt}) we have 
\begin{equation}
  \label{eq:sca}
  p_t(x)=t^{-\frac{d}{\alpha}}p_1(t^{-\frac{1}{\alpha}}x)\,.
\end{equation}
It is well-known  \cite{MR2569321} that 
$p_1(x)
\approx1\land |x|^{-d-\alpha}$, hence
\begin{equation}\label{eq:oppt}
p_t(x)
\approx t^{-d/\alpha}\land \frac{t}{|x|^{d+\alpha}}
\,,\quad t>0,\,x\in \Rd\,.
\end{equation}
We denote 
$$
p(t,x,y)=p_t(y-x), \quad t>0,\ x,y\in \Rd.
$$
Clearly, $p$ is symmetric,
\begin{equation*}\label{pt}
   p(t,x, y) = p(t,y, x) \,,
\end{equation*}   
and satisfies the Chapman-Kolmogorov equations:
\begin{equation*}
     \int_{\Rd} p(s,x, y)p(t,y, z)dy = p(s+t,x, z), \quad x, z \in  \Rd,\, s, t > 0.
\end{equation*}     
We denote, as usual, $P_t g(x) = \int_{\Rd} p(t,x, y) g(y){\rm d}y$.
The fractional Laplacian extends to the generator of the semigroup $\{P_t\}_{t\geq0}$ on many Banach spaces, see, e.g., \cite{2015arXiv150707356K}.

\subsection{Schr\"odinger perturbations by Hardy potential}
For $\alpha < d$ and $\beta \in (0, d)$, 
let 
$$f(t) = c_1t^{(d-\alpha -\beta)/\alpha },\quad t>0.$$ 
Here $c_1\in (0, \infty)$ is a normalizing constant so chosen that 
\begin{equation}\label{h_beta}
   h_\beta(x) := \int_0^\infty f(t) p_t(x) {\rm d}t =|x|^{-\beta}, \quad x\in \Rd.
\end{equation}
The existence of such $c_1$ is a consequence of \eqref{eq:sca}.
The exact value of $c_1$ does not affect the subsequent definition of $q_\beta$.
For $\beta \in (0, d-\alpha)$ we let
\begin{equation*}\label{def_q}
     q_\beta(x) = \frac{1}{h_\beta(x)}\int_0^{\infty }f'(t) p_t(x) {\rm d}t,\quad x\in \Rd,
\end{equation*}
By \cite[Section~4]{MR3460023}\footnote{Note that the exponent $(d-\alpha -\beta)/\alpha$ in the definition of $f$ is denoted $\beta$ in \cite[Corollary~6]{MR3460023}.},
\begin{equation*}\label{q_beta}
  q_\beta(x) = \kappa_\beta |x|^{-\alpha },
\end{equation*}
where
$$
\kappa_\beta=\frac{2^\alpha\Gamma(\frac{\beta+\alpha}{2}) \Gamma(\frac{d-\beta}{2})}{\Gamma(\frac{\beta}{2})\Gamma(\frac{d-\beta-\alpha}{2})}.
$$
Put differently, 
\begin{equation}\label{eq:toz}
\kappa_\beta |x|^{-\alpha-\beta} = \int_0^\infty f'(t) p_t(x)  {\rm d}t.
\end{equation}
In  \cite{MR3460023} the above objects have applications to both in the integral test for Schr\"odinger semigroups  \cite[Theorem~1]{MR3460023}
and in Hardy inequality for Dirichlet forms  \cite[Theorem~2]{MR3460023}. Similarly
they play a double  role here:
in 
the integral analysis method and in Davies' method. We further note that
 \eqref{eq:toz} is an integral analogue of \cite[(3.6)]{2017arXiv170503310B}.
The function $\beta \mapsto \kappa_\beta$ is increasing on $(0,(d-\alpha)/2]$ and decreasing on $[(d-\alpha)/2,d-\alpha)$, see  \cite[Proof of Proposition~5]{MR3460023}. Furthermore, $\kappa_\beta = \kappa_{d-\alpha-\beta}$. The maximal or {\it critical} value of $\kappa_\beta$ is $\kappa^*=\kappa_{(d-\alpha)/2}$.
For convenience, let $\kappa_0 = 0$.
In what follows, we fix $\delta \in [0,(d-\alpha)/2]$. Let
\begin{equation}\label{eq:dk}
h(x)=h_\delta(x)=|x|^{-\delta},\quad 
\kappa = \kappa_\delta,\quad
q(x) = q_\delta(x)=\kappa |x|^{-\alpha}.
\end{equation}
The notation will be used throughout the paper.
\begin{figure}
\caption{The function $\beta\mapsto \kappa_\beta$.}
{\protect\label{fig:k}}
\includegraphics[width=8cm]{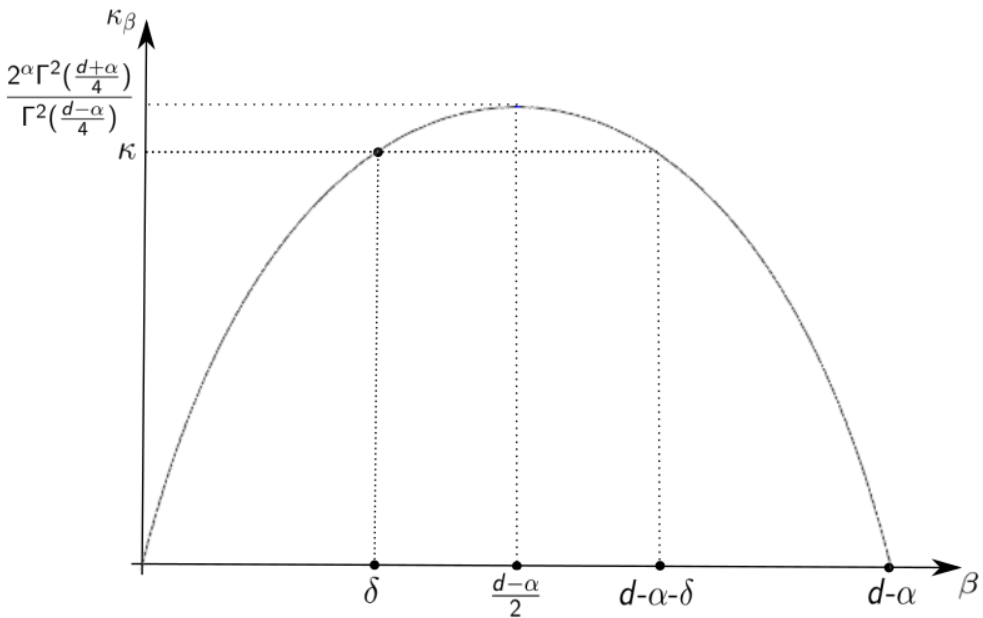}
\end{figure}
Also, let $\delta^*=(d-\alpha)/2$ and $q^*=q_{\delta^*}$.

We define the Schr\"odinger perturbation of $p$ by $q=q_\delta$:
\begin{equation}\label{def_p_tilde}
   \tilde p = \sum_{n=0}^{\infty }p_n.
\end{equation}
Here for $t>0$ and $x,y\in \Rd$ we let $p_0(t,x, y) = p(t,x, y)$ and
\begin{align}\label{pn}
     p_n(t,x, y) &= \int_0^t \int_{\Rd} p(s,x, z) q(z)p_{n-1}(t-s, z, y)  {\rm d}z  {\rm d}s\\
     &=\int_0^t \int_{\Rd} p_{n-1}(s,x, z) q(z)p(t-s, z, y)  {\rm d}z  {\rm d}s
     , \quad n \geq 1.\notag
\end{align}     
If $q=q^*$, then we may write $\tp^*$ for $\tp$. 
From the general theory \cite{MR3460023}, $\tilde p$ is a symmetric transition density, i.e. the following Chapman-Kolmogorov equation holds:
\begin{align}\label{eq:ChK}
\int_\Rd \tp(s,x,z) \tp(t,z,y)\, dz = \tp(t+s,x,y)\,.
\end{align} 
Clearly, 
the following Duhamel's formula holds,
\begin{align}\label{eq:Df}
\tp(t,x,y) &= p(t,x,y) + \int_0^t \int_{\bbfR^d} p(s,x,z) q(z) \tp(t-s,z,y) dz ds\\
&= p(t,x,y) + \int_0^t \int_{\bbfR^d} \tp(s,x,z) q(z) p(t-s,z,y) dz ds,\quad t>0,\ x,y\in \bbfR^d. \label{eq:Df2}
\end{align}
The finiteness of $\tilde p(t,x, y)$ for all $t>0$, $x\neq 0$ and $a.e.$~$y\in \Rd$ is secured as follows.
\begin{lemma}\cite[Theorem~1]{ MR3460023} \label{pt_h}
We have $\int_{\Rd} \tilde p(t,x, y)h(y){\rm d}y \leq  h(x)$.
\end{lemma}
In the sense of Lemma~\ref{pt_h}, $h$ is supermedian. The scaling  of $\tp$ is the same as that of $p$.
\begin{lemma}\label{ss th}
The kernel $\tilde p$ is self-similar,
\begin{equation}\label{ss form}
  \tilde p(t,x, y)=t^{-\frac{d}{\alpha }}\tilde p\big(1, xt^{-\frac{1}{\alpha }}, yt^{-\frac{1}{\alpha }}\big), \qquad  t>0,  \ x,y \in \bbfR^d.
\end{equation}
\end{lemma}

\begin{proof}
By the definition \rf{def_p_tilde} of 
$\tilde p$, it is enough to show that for all $n\geq 0$ and $t>0$,
\begin{equation}\label{ss form2}
    p_n(t,x,y) = t^{-\frac{d}{\alpha }}p_n\big(1, xt^{-\frac{1}{\alpha }}, yt^{-\frac{1}{\alpha }}\big),
\end{equation}
or, equivalently,  for all $n\geq 0$ and $t,u>0$,
\begin{equation*}\label{ss form3}
   p_n(tu,x,y)=t^{-\frac{d}{\alpha }}p_n\big(u, xt^{-\frac{1}{\alpha }}, yt^{-\frac{1}{\alpha }}\big).
\end{equation*} 
By \eqref{eq:sca}, the equality \rf{ss form2} holds true for $n=0$. Let $n \geq 0$ be an integer, and assume that \rf{ss form2} holds for $n$. By this, 
the definition \rf{pn} applied to $n+1$ instead of $n$ and by the change of variables $s = tu$ and $w = zt^{-\frac{1}{\alpha }}$,
\begin{align*}
    p_{n+1}(t,x, y) &= \int_0^t \int_{\bbfR^d} p(s,x, z) q(z)p_n(t-s,z, y)  dz  ds \\
    &= t\int_0^1 \int_{\bbfR^d} p(tu,x, z) q(z)p_n(t(1-u),z, y)  dz  du \\
 & = t^{1-\frac{2d}{\alpha }}\int_0^1 \int_{\bbfR^d} p(u,xt^{-\frac{1}{\alpha }}, zt^{-\frac{1}{\alpha }}) q(z)p_n(1-u,zt^{-\frac{1}{\alpha }}, yt^{-\frac{1}{\alpha }})  dz  du \\
    & = t^{-\frac{d}{\alpha }}\int_0^1 \int_{\bbfR^d} p(u,xt^{-\frac{1}{\alpha }}, w) \kappa |w|^{-\alpha } p_n(1-u,w, yt^{-\frac{1}{\alpha }})  dw  du\\
    &=t^{-\frac{d}{\alpha }}p_{n+1}\big(1, xt^{-\frac{1}{\alpha }}, yt^{-\frac{1}{\alpha }}\big).
\end{align*}
By induction, we obtain \rf{ss form2} for all $n \ge 0$, and so  \rf{ss form} follows.
\end{proof}
We need to go beyond the integrals in \cite{MR3460023}, as follows.
\begin{lemma}\label{lem:h}
Let $\beta \in (0,d)$. Then,
\begin{align} 
\int_{\RR^d} p(s,x,z) |z|^{-\beta} dz &\approx  (s^{-\beta/\alpha} \land |x|^{-\beta}), \label{eq1:h}\\
\int_0^t \int_{\RR^d} p(s,x,z) |z|^{-\beta} dz ds &\approx 
\begin{cases}
 t(t^{-\beta/\alpha} \land |x|^{-\beta}) & \mbox{for } \beta<\alpha, \\
\ln(1 + t|x|^{-\alpha}) & \mbox{for } \beta =\alpha, \\
 (t \land |x|^\alpha)|x|^{-\beta} & \mbox{for } \beta>\alpha.
\end{cases}\label{eq2:h}
\end{align}
\end{lemma}

\begin{proof}
Note that by \eqref{h_beta},
$$
h_\beta(s,x) := \int_{\RR^d} p(s,x,z) |z|^{-\beta} dz = c_\beta \int_0^\infty p(s+u,x) u^{(d-\alpha-\beta)/\alpha} du 
$$
for some constant $c_\beta$.
By scaling,
\begin{align*} 
h_\beta(s,x) &= c_\beta |x|^{-d} \int_0^\infty p(s|x|^{-\alpha}+r,1) |x|^{d-\alpha-\beta} r^{(d-\alpha-\beta)/\alpha} |x|^\alpha dr\\
&= c_\beta |x|^{-\beta} \int_0^\infty p(s|x|^{-\alpha}+r,1)  r^{(d-\alpha-\beta)/\alpha}  dr.
\end{align*}
For $|x|^\alpha < s$, by \eqref{eq:oppt} we have $p(s|x|^{-\alpha}+r,1) \approx (s|x|^{-\alpha}+r)^{-d/\alpha}$,
hence
\begin{align*} 
h_\beta(s,x) &\approx |x|^{-\beta} \int_0^\infty (s|x|^{-\alpha})^{-d/\alpha} (1+u)^{-d/\alpha} (s|x|^{-\alpha})^{(d-\alpha-\beta)/\alpha}  u^{(d-\alpha-\beta)/\alpha} (s|x|^{-\alpha}) du\\
&= s^{-\beta/\alpha} \int_0^\infty (1+u)^{-d/\alpha} u^{(d-\alpha-\beta)/\alpha} du = cs^{-\beta/\alpha}.
\end{align*}
For $|x|^\alpha \ge s$, again by \eqref{eq:oppt} we have $$c^{-1}p(r,1)\leq p(s|x|^{-\alpha}+r,1) \leq c( p(r,1)+p(r+1,1)).$$ Consequently $h_\beta(s,x) \approx |x|^{-\beta}$.
Now, \eqref{eq2:h} follows by integrating \eqref{eq1:h}.
\end{proof}
Let $\tilde P_t f(x) = \int_\Rd \tilde p(t,x, y) f(y) dy$.
\begin{proposition}\label{prop:SCP}If $0<\kappa\leq\kappa^*$, then
$\{\tilde P_t\}_{t>0}$ is a strongly continuous semigroup of contractions on $L^2(\Rd)$.
\end{proposition}
\begin{proof}
By the symmetry of $\tilde p$, Lemma \ref{pt_h} and the Schur test, $\tilde P_t$ is a contraction on $L^2(\Rd)$ for every $t>0$.  
To prove the strong continuity, let $g \in L^2(\Rd)$ and $\varphi \in C_c^{\infty } ( \Rd)$.
We have
\begin{equation*}
   \|\tilde P_t g- g \|_{L^2} \leq \| \tilde P_t g- \tilde P_t \varphi\|_{L^2} + \| \tilde P_t \varphi - \varphi\|_{L^2} + \| \varphi - g\|_{L^2}
   \leq\| \tilde P_t \varphi - \varphi\|_{L^2} + 2\| \varphi - g\|_{L^2} .
\end{equation*}
By the density of $C^\infty_c(\Rd)$ in $L^2(\Rd)$, it suffices to verify that $\| \tilde P_t \varphi - \varphi\|_{L^2}\to 0$ as $t\to 0$.
To this end we consider \eqref{eq:Df} and let $T_t$ be the integral operator with the kernel
$$
\int_0^t \int_{\bbfR^d} p(s,x,z) q(z) \tp(t-s,z,y) dz ds.
$$ 
Since $\{P_t\}_{t\geq0}$ is strongly continuous on $L^2(\Rd)$, it is enough to prove that
\begin{equation}\label{eq:T_t0}
\lim_{t\to 0^+}\|T_t\varphi\|_{L^2}=0,
\end{equation}
for every non-negative function $\varphi\in C^\infty_c(\Rd)$. Since the kernel of $T_t$ is maximal for $\delta=(d-\alpha)/2$, it suffices to consider 
this case. 
Then there is a constant $C$, depending on $\varphi$, such that
$\varphi\leq C h^*$. By Lemma \ref{pt_h} and Lemma \ref{lem:h}, 
\begin{align*}
\|T_t\varphi\|^2_{L^2}&\leq C\|T_t h^*\|^2_{L^2} \leq C\int_{\Rd}\left(\int^t_0\int_{\Rd}p(s,x,z)q(z)h^*(z)dzds\right)^2dx\\&\approx \int_{\RR^d}(t^2 \land |x|^{2\alpha})|x|^{-d-\alpha} dx,
\end{align*}
and \eqref{eq:T_t0} follows by the dominated convergence theorem.
\end{proof}

\section{Integral analysis}\label{sec:ia}
In this section we propose a new method of estimating the heat kernel of  Schr\"odinger heat kernels.
The method picks up on the  ideas of Bogdan, Dyda and Kim \cite[Section~2]{MR3460023} and develops a suitable integral calculus to handle $\tp$ by means of the Duhamel's formula.
By our conventions, $\tp=p$ if $\delta=0$.
We will study the integrals
$$
\int_{\bbfR^d} \tp(t,x,y) |y|^{-\beta} dy, \qquad 0\le\beta<d-\alpha.
$$
For $\beta=\delta$, from Lemma \ref{pt_h} we have
\begin{align}\label{eq:tp_bas_est}
	\int_{\bbfR^d} \tp(t,x,y) |y|^{-\delta} dy \le |x|^{-\delta}, \qquad t>0,\; x \in \Rd\,.
\end{align}

One of our main results is the following.
\begin{theorem}\label{thm:1}
For $\delta \in [0,\frac{d-\alpha}{2}]$, 
\begin{align}\label{eq:2}
	\int_{\bbfR^d} \tp(t,x,y)  |y|^{-\delta} dy = |x|^{-\delta}, \qquad t>0,\; x \in \Rd\,.
\end{align}
Furthermore for $\beta \in [0, d-\alpha-\delta)$, $t>0$, $x \not=0$,
\begin{align}\label{eq:1}
	\int_{\bbfR^d} \tp(t,x,y)  |y|^{-\beta} dy = |x|^{-\beta} + (\kappa - \kappa_\beta) \int_0^t \int_{\bbfR^d} \tp(s,x,y) |y|^{-\beta-\alpha}\,dy\,ds<\infty.
\end{align}
\end{theorem}
The proof of Theorem~\ref{thm:1} is given below after a sequence of auxiliary results.
Noteworthy, in \eqref{eq:1} and similar formulas we have two types of integrations: in space-time up to the terminal time $t$, and in space only, at the terminal time $t$. We also note that \eqref{eq:2} is an immediate consequence of \eqref{eq:1}, except in the critical case $\delta=\delta^*=(d-\alpha)/2$.

\begin{proposition}\label{lem:int_tpinfty}
There is a constant $M \ge 1$ such that for all $t \in (0,\infty)$ and $x \in \bbfR^d \setminus \{0\}$,
\begin{equation*}
	\int_{\bbfR^d} \tp(t,x,y) dy \le M (1+t^{\delta/\alpha}|x|^{-\delta}).
\end{equation*}
\end{proposition}
\begin{proof}
By Lemma \ref{ss th}, it suffices to consider $t=1$.
Fix $x \not= 0$. We note that for every $R>0$, by \eqref{eq:Df2} and \eqref{eq:ChK},
\begin{align}
\tp(1,x,y) &= p(1,x,y) + \int_0^1 \int_{|z|<R}\tp(s,x,z) q(z) p(1-s,z,y)\,dz\,ds \label{eq1:int_tpinfty}\\
 &+  \int_0^1 \int_{|z|>R}\tp(s,x,z) q(z) p(1-s,z,y)\,dz\,ds \notag\\
&\le p(1,x,y) + \int_0^1 \int_{|z|<R}\tp(s,x,z) q(z) p(1-s,z,y)\,dz\,ds \notag\\
 &+  \frac{\kappa}{R^\alpha}\int_0^1 \int_{|z|>R}\tp(s,x,z) p(1-s,z,y)\,dz\,ds \notag\\
&\le p(1,x,y) + \int_0^1 \int_{|z|<R}\tp(s,x,z) q(z) p(1-s,z,y)\,dz\,ds +  \frac{\kappa}{R^\alpha}\tp(1,x,y)\,. \notag
\end{align}
Hence, for $R\ge(2\kappa)^{1/\alpha}$,
\begin{align}\label{eq2:int_tpinfty}
	\tp(1,x,y) \le 2p(1,x,y) + 2\int_0^1 \int_{|z|<R}\tp(s,x,z) q(z) p(1-s,z,y)\,dz\,ds.
\end{align}
On the other hand, by the equality \eqref{eq1:int_tpinfty} we have
\begin{align}\label{eq3:int_tpinfty}
	\tp(1,x,y) \ge p(1,x,y) + \int_0^1 \int_{|z|<R}\tp(s,x,z) q(z) p(1-s,z,y)\,dz\,ds.
\end{align}
Integrating \eqref{eq2:int_tpinfty} and \eqref{eq3:int_tpinfty} with respect to $dy$, we get
\begin{align*}\label{eq4:int_tpinfty}
	\int_{\bbfR^d} \tp(t,x,y) dy \approx 1 + \int_0^1\int_{|z|<R}\tp(s,x,z) q(z) \,dz\,ds\,.
\end{align*}
By \eqref{eq:oppt}, there is $c \in (0,1)$ such that for $|z|<R$ and $s \in (0,1)$, 
$$
c \le \int_{B(0,2R)} p(1-s,z,y) dy \le 1\,.
$$
By \eqref{eq:Df2} and \eqref{eq:tp_bas_est},
\begin{align*}
\int_0^1\int_{|z|<R}\tp(s,x,z) q(z) \,dz\,ds &\le \frac{1}{c} \int_{B(0,2R)}\int_0^1 \int_{|z|<R} \tp(s,x,z) q(z) p(1-s,z,y)\,dz\,ds\,dy \\
& \le \frac{1}{c} \int_{B(0,2R)}\tp(1,x,y)\,dy \le \frac{(2R)^\delta}{c} \int_{B(0,2R)}\tp(1,x,y) |y|^{-\delta}\,dy \\
&\le \frac{(2R)^\delta}{c}|x|^{-\delta}<\infty.
\end{align*}
\end{proof}

The special case $\kappa =0$ of Theorem \ref{thm:1} can already be verified as follows. 

\begin{lemma}\label{lem:p_id} For $\beta \in [0,d-\alpha)$, $t>0$ and $x\in \bbfR^d$, 
$$
 \kappa_\beta \int_0^t \int_{\bbfR^d} p(s,x,y) |y|^{-\beta-\alpha} \,dy \,ds + \int_{\bbfR^d} p(t,x,y)|y|^{-\beta} dy = |x|^{-\beta}.
$$
\begin{proof}
For $\beta =0$, we simply get $1=1$, so suppose $\beta>0$ and $x\not=0$. Then, by \eqref{eq:toz},
\begin{align*}
& \kappa_\beta \int_0^t \int_{\bbfR^d} p(s,x,y) |y|^{-\beta-\alpha} \,dy \,ds  =  \int_0^t \int_0^\infty p_{s+r}(x) f'(r) dr ds\\
&= -\int_0^t \int_0^\infty \frac{\partial}{\partial s} p_{s+r}(x) f(r) dr ds
=  -\int_0^\infty [p_r(x) - p_{t+r}(x)] f(r) dr \\
&= |x|^{-\beta} - \int_{\bbfR^d} p(t,x,y)|y|^{-\beta} dy,
\end{align*}
as needed.
For $x=0$  the iterated integral in the statement diverges
by \eqref{eq:oppt}. 
\end{proof}
\end{lemma}
Although we do not need this observations later on, if we let $\beta \to 0$ in 
$$
\int_0^t \int_{\bbfR^d} p(s,x,y) |y|^{-\beta-\alpha} \,dy \,ds = \frac{1}{\kappa_\beta} \int_{\bbfR^d} p(t,x,y) (|x|^{-\beta} - |y|^{-\beta}) dy,
$$
then for $t>0$, $x\in \Rd\setminus\{0\}$, we get
\begin{equation}\label{eq:log}
\frac{\Gamma(\frac{\alpha}{2}) \Gamma(\frac{d}{2})}{2^{1-\alpha}\Gamma(\frac{d-\alpha}{2})}\int_0^t \int_{\bbfR^d} p(s,x,y) |y|^{-\alpha} \,dy \,ds =  \int_{\bbfR^d} p(t,x,y)\ln(|y|) dy - \ln(|x|). 
\end{equation}

Back to the proof of Theorem \ref{thm:1}, we recall that the functions $p_n$ are defined in \eqref{pn}.

\begin{lemma}\label{lem:int_pn_est}
Let $0 < \beta < d-\alpha$ and $n \ge 0$. For all $t>0$  and $x\not=0$, 
\begin{align}
\int_0^t \int_{\bbfR^d} p_n(s,x,y) |y|^{-\beta-\alpha} \,dy \,ds \le \frac{\kappa^n}{\kappa_\beta^{n+1}}|x|^{-\beta} < \infty . \label{eq:int_pn_est}
\end{align}
\end{lemma}
\begin{proof}
Let $x \in \bbfR^d$. For $n=0$, \eqref{eq:int_pn_est} follows from Lemma \ref{lem:p_id}. Suppose \eqref{eq:int_pn_est} holds for some $n \ge 0$. By Lemma \ref{lem:p_id} and induction,
\begin{align*}
&\int_0^t \int_{\bbfR^d} p_{n+1}(s,x,y) |y|^{-\beta-\alpha}\, dy\, ds \\
&= \int_0^t \int_{\bbfR^d} \int_0^s \int_{\bbfR^d} p_{n}(r,x,w) q(w) p(s-r,w,y)|y|^{-\beta-\alpha} \,dy \,dr \,dw \,ds \\
&= \int_0^t \int_{\bbfR^d} \int_0^{t-r} \int_{\bbfR^d} p_{n}(r,x,w) q(w) p(s,w,y)|y|^{-\beta-\alpha} \,dy \,ds \,dw \,dr \\
&\le  \int_0^t \int_{\bbfR^d}  p_{n}(r,x,w) q(w) \frac{1}{\kappa_\beta}|w|^{-\beta} \,dw \,ds =  \frac{\kappa}{\kappa_\beta} \int_0^t \int_{\bbfR^d}  p_{n}(r,x,w) |w|^{-\beta-\alpha} \,dw \,ds \\
&\le \frac{\kappa^{n+1}}{\kappa_\beta^{n+2}}|x|^{-\beta},
\end{align*} 
which is finite if $x\not=0$.
\end{proof}
\begin{corollary}\label{cor:tp_bas_est2}
For $\delta<\beta<d-\alpha-\delta$, $t>0$ and $x\not=0$, we have
$$
\int_0^t \int_{\bbfR^d} \tp(s,x,y) |y|^{-\beta-\alpha} \,dy \,ds \le \frac{|x|^{-\beta}}{\kappa_\beta - \kappa}<\infty .
$$
\end{corollary}
\begin{proof}
If $\delta<\beta<d-\alpha-\delta$, then $\kappa_\beta > \kappa$. 
Lemma \ref{lem:int_pn_est} and \eqref{def_p_tilde} yield the result.
\end{proof}

\begin{lemma}\label{lem:pn_id}
Let $0\le \beta < d-\alpha$ and $n \ge 1$. For all $t>0$  and $x\not=0$, we have
$$
\kappa_\beta\int_0^t\! \int_{\bbfR^d}\! p_n(s,x,y) |y|^{-\beta-\alpha} \,dy \,ds + \int_{\bbfR^d} \!p_n(t,x,y) |y|^{-\beta} dy = \kappa\int_0^t\! \int_{\bbfR^d} \!p_{n-1}(s,x,y) |y|^{-\beta-\alpha} dy ds.
$$
\end{lemma}
\begin{proof}
By Lemma \ref{lem:p_id}, 
\begin{align*}
&\kappa_\beta\int_0^t \int_{\bbfR^d} p_n(s,x,y) |y|^{-\beta-\alpha} dy ds  \\
&= \kappa_\beta\int_0^t \int_{\bbfR^d} \int_0^s \int_{\bbfR^d} p_{n-1}(u,x,w) q(w)  p(s-u,w,y) |y|^{-\beta-\alpha} \,dw\,du \,dy \,ds  \\
&= \kappa_\beta\int_0^t \int_{\bbfR^d} \int_0^{t-u} \int_{\bbfR^d} p_{n-1}(u,x,w) q(w)  p(s,w,y) |y|^{-\beta-\alpha} \,dy \,ds \,dw \,du \\
&= \int_0^t \int_{\bbfR^d}  p_{n-1}(u,x,w) q(w)  \left(|w|^{-\beta} - \int_{\bbfR^d} p(t-u,w,y)|y|^{-\beta} dy\right)  \,dw \,du \\
& =\kappa\int_0^t   \int_{\bbfR^d}  p_{n-1}(u,x,w) |w|^{-\beta-\alpha} \,dw \,du - \int_{\bbfR^d} p_n(t,x,y) |y|^{-\beta} dy .
\end{align*}
The integrals are convergent by Lemma \ref{lem:int_pn_est}. 
\end{proof}
Summing up the equalities from Lemma \ref{lem:pn_id} and \ref{lem:p_id} and using Corollary \ref{cor:tp_bas_est2} yields \eqref{eq:1} in Theorem \ref{thm:1} for $\delta <\beta <d-\alpha -\delta$. The full range of $\beta$ needs additional preparation.

\begin{corollary}\label{cor:pn_id}
For $N\in \bbfN$, $t>0$, $x\not=0$ and $\beta \in [0, d-\alpha)$,
\begin{align}
& \int_{\bbfR^d} \sum_{n=0}^N p_n(t,x,y) |y|^{-\beta} dy + \kappa_\beta \int_0^t \int_{\bbfR^d} p_N(s,x,y)  |y|^{-\beta-\alpha} \,dy \,ds \notag \\
&= |x|^{-\beta} + (\kappa - \kappa_\beta)\int_0^t \int_{\bbfR^d} \sum_{n=0}^{N-1} p_n(s,x,y) |y|^{-\beta-\alpha} dy ds. \label{eq:sum_pn_id}
\end{align}
\end{corollary}
\begin{proof}
For $n=1,2,..., N$ we sum up the identity from Lemma \ref{lem:pn_id}, add the identity from Lemma \ref{lem:p_id}, and use the finiteness from Corollary \ref{cor:tp_bas_est2} to justify subtraction.
\end{proof}
\begin{proof}[Proof of Theorem \ref{thm:1}]
We start with \eqref{eq:1}. In view of Corollary \ref{cor:pn_id}, 
we only need to show 
\begin{align}\label{eq:pNto0}
\lim_{N \to \infty}\int_0^t \int_{\bbfR^d} p_N(s,x,y)  |y|^{-\beta-\alpha} dy ds = 0.
\end{align}
Indeed, if we have \eqref{eq:pNto0}, then letting $N \to \infty$ in \eqref{eq:sum_pn_id}, by the monotone convergence theorem, we get
\begin{align*}
\int_{\bbfR^d} \tp(t,x,y) |y|^{-\beta} dy = |x|^{-\beta} + (\kappa - \kappa_\beta)\int_0^t \int_{\bbfR^d} \tp(s,x,y) |y|^{-\beta-\alpha} dy ds.
\end{align*}
The integrals are convergent for $\beta \in [0,\delta]$. We apply \eqref{eq:tp_bas_est} and Proposition \ref{lem:int_tpinfty} to the left integral, and for $\beta \in (\delta, d-\alpha -\delta)$ we apply Corollary \ref{cor:tp_bas_est2} to the integral on the right.

Proving \eqref{eq:pNto0} is nontrivial. Of course, by Lemma \ref{pt_h}, for $t,|x| >0$ we have $\tp(t,x,\cdot) <\infty$ $a.e.$ Hence, $p_N(t,x,y) \to 0$ as $N \to \infty$. To apply the dominated convergence theorem it suffices to give an integrable majorant for $p_N(t,x,y)$ in \eqref{eq:pNto0}.

First we let $\beta \in(\delta,d-\alpha-\delta)$. Since $p_N(t,x,y) \le \tp(t,x,y)$, Corollary \ref{cor:tp_bas_est2} yields \eqref{eq:pNto0}.
Now let $\beta < \delta$. By Lemma \ref{lem:int_pn_est} with $\beta = \delta$,  \eqref{eq:Df2} integrated in $dy$ 
and Proposition \ref{lem:int_tpinfty},
\begin{align*}
\int_0^t \!\int_{\bbfR^d} \! p_N(s,x,y)  |y|^{-\beta-\alpha} dy ds & \le \int_0^t \! \int_{\bbfR^d} \! p_N(s,x,y)  |y|^{-\delta-\alpha} dy ds + \int_0^t \!\int_{\bbfR^d} \!\tp(s,x,y)|y|^{-\alpha}   dy ds \\
& \le \frac{1}{\kappa}\left(|x|^{-\delta} + \int_{\bbfR^d} \! \tp(t,x,y)\,dy\right)<\infty .
\end{align*}
Hence, by Corollary \ref{cor:pn_id} with $\beta<\delta$,
\begin{align*}
  (\kappa - \kappa_\beta )\!\int\limits_0^t \!\int\limits_{\bbfR^d} \sum\limits_{n=0}^{N-1}  \!p_n(s,x,y)|y|^{-\beta-\alpha} dy ds \! \le \! \int\limits_{\bbfR^d}\!\tp(s,x,y)|y|^{-\beta} dy \!+\!\frac{\kappa_\beta}{\kappa}\left(\!|x|^{-\delta}\!+\!\int\limits_{\bbfR^d}\!\tp(t,x,y) dy\!\right).
\end{align*}
By passing to the limit and using $|y|^{-\beta}\le |y|^{-\delta}+1$ we get
\begin{align*}\label{eq:1tmp1}
(\kappa - \kappa_\beta)\! \int\limits_0^t\! \int\limits_{\bbfR^d}\! \tp(s,x,z)  |z|^{-\beta-\alpha} dz ds & \!\le\! \int\limits_{\bbfR^d} \!\tp(t,x,y)(1\!+\!|y|^{-\delta})\,dy\! +\! \frac{\kappa_\beta}{\kappa}\left(\!|x|^{-\delta}\!+\!\int\limits_{\bbfR^d}\!\tp(t,x,y) dy\!\right),
\end{align*}
which is finite by \eqref{eq:tp_bas_est}  and Lemma \ref{lem:int_tpinfty}. Since $p_N(t,x,y) \le \tp(t,x,y)$, we get \eqref{eq:pNto0}.

Next, let $\beta = \delta< \delta^* =\frac{d-\alpha}{2}$.
Clearly, $p_N\le \tp \le \tp^*$. By the previous case we get
\begin{align*}
 \int_0^t \int_{\bbfR^d} \tp^*(s,x,z) |z|^{-\delta-\alpha} \,dz \,ds <\infty,
\end{align*}
hence \eqref{eq:pNto0} follows. 
We see that \eqref{eq:1} is fully proven. This yields \eqref{eq:2} for $\delta \in[0,\frac{d-\alpha}{2})$. 
For $\delta = \delta^*$ we let $0\le \beta < \delta$. As usual,  $|y|^{-\beta}\le |y|^{-\delta}+1$, and 
$$
  \int_{\bbfR^d} \tp(t,x,y) (|y|^{-\delta} +1) dy < \infty.
$$
By \eqref{eq:1},
$$
\int_{\bbfR^d} \tp(t,x,y)|y|^{-\beta} dy \ge |y|^{-\beta},
$$
hence, letting $\beta \to \delta$, by the dominated convergence theorem we get  
$$
\int_{\bbfR^d} \tp(t,x,y) |y|^{-\delta} dy \ge |x|^{-\delta}.
$$
In view of  \eqref{eq:tp_bas_est} we get \eqref{eq:2} for $\delta = \delta^*$ and $\tp=\tp^*$, too.
\end{proof}

We denote
\begin{equation*}\label{H_func}
   H (t,x) = t^{ \frac{\delta}{\alpha }} |x|^{-\delta} + 1,\qquad  t>0, \quad x\in \Rd.
\end{equation*}
We also let $H(x)=H(1,x)$. Thus, 
$$H(x) = |x|^{-\delta} +1, \qquad x\in \Rd.$$ 
Proposition~\ref{lem:int_tpinfty} and Theorem~\ref{thm:1} yield that $H(t,x)$ is nearly supermedian, as follows.
\begin{corollary}\label{cor:sh}
For all $t \in (0,\infty)$ and $x \in \bbfR^d \setminus \{0\}$,
\begin{align*}\label{eq:sh}
	\int_{\bbfR^d} \tp(t,x,y) H(t,y)dy \le (M+1) H(t,x).
\end{align*}
\end{corollary}


\section{Estimates and continuity of the Schr\"odinger heat kernel}\label{sec:pt1}

\subsection{Upper bound  of the heat kernel}

As usual, $\delta \in [0, (d-\alpha )/2]$. 
By Corollary \ref{cor:sh}, we immediately get the following result.

\begin{lemma}\label{cor:estintHbis}
If $\beta \in [0, \delta]$, then
$$
  \int_{\bbfR^d} \tp(1,x,z) |z|^{-\beta} dz \le CH(x) .
$$
\end{lemma}

Let $g(y)=|y|^\alpha/(2^{\alpha+1}\kappa)$. Note that $|y| \le 2(2\kappa)^{1/\alpha}$ is equivalent to $g(y) \le 1$.

\begin{lemma}\label{lem:tp_est_prel}
There exists a constant $c$ such that for $y\in\Rd$ and $|x|\leq 2(2\kappa)^{1/\alpha}$,
\begin{align*}\tp(1,x,y)\leq& c H(x)p_1(y) +c\int^{1/2}_0\int_{2|z|\leq |y|}\tp(1-s,x,z)q(z)p_s(y)dzds\\&+2\int^1_{g(y)\wedge 1}\int_{2|z|>|y|} \tp(1-s,x,z)q(z)p(s,z,y)dzds.
\end{align*}
\end{lemma}

\begin{proof}

 We will estimate
\begin{align*}
\tp(1,x,y) &=  p(1,x,y) + \mathrm{I}_1 + \mathrm{I}_2 + \mathrm{I}_3,
\end{align*}
where
\begin{align*}
\mathrm{I}_1 &= \int_0^{g(y)\wedge 1} \int_{2|z| > |y|} \tp(1-s,x,z) q(z) p(s,z,y) dz ds,\\
\mathrm{I}_2 &= \int^1_{g(y)\wedge 1} \int_{2|z| > |y|} \tp(1-s,x,z) q(z) p(s,z,y) dz ds,\\
\mathrm{I}_3& = \int_{0}^1 \int_{2|z| \le |y|} \tp(1-s,x,z) q(z) p(s,z,y) dz ds. 
\end{align*}
Clearly, by \eqref{eq:ChK},
$$\mathrm{I}_1\leq q(|y|/2)\int^{g(y)\wedge 1}_0\int_{\Rd} \tp(1-s,x,z)\tp(s,z,y)dzds\leq\frac{1}{2}\tp(1,x,y).$$
Due to Proposition \ref{lem:int_tpinfty}, by integrating \eqref{eq:Df2} against $y$, we get
\begin{equation}\label{eq:tp_q}
\int_{0}^1 \int_{\Rd} \tp(1-s,x,z) q(z) dz\leq M H(x).
\end{equation}
By \eqref{eq:oppt}, we have $p(s,z,y)\approx p_s(y)$ for $|z|\leq |y|/2$ and $p_s(y)\approx p_1(y)$ for $s\in[1/2,1]$. Hence,
\begin{align*}
\mathrm{I}_3\approx &\int_{0}^1 \int_{2|z| \le |y|} \tp(1-s,x,z) q(z) p_s(y) dz ds\\
\approx& \int_{0}^{1/2} \int_{2|z| \le |y|} \tp(1-s,x,z) q(z) p_s(y) dz ds+p_1(y)\,M\,H(x).
\end{align*}
Since $|x|\leq 2(2\kappa)^{1/\alpha}$, by \eqref{eq:oppt} we have $p(1,x,y)\approx p_1(y)$,  which ends the proof.

\end{proof}

\begin{lemma}\label{lem:ESTtp0}
There is $C>0$ such that for $|y| \ge 2(2\kappa)^{1/\alpha}$ and $x\in\Rd$,  
\begin{equation*}
	\tp(1,x,y) \le  CH(x)p(1,x,y).
\end{equation*}
\end{lemma}

\begin{proof}
By \eqref{eq:oppt}, $p_s(y)\le c p_1(y)$ for $s\leq 1$. Hence by \eqref{eq:tp_q},
\begin{align*}
\int^{1/2}_0\int_{2|z|\leq |y|}\tp(1-s,x,z)q(z)p_s(y)dzds \leq  H(x) |y|^{-d-\alpha}.
\end{align*}
Since $g(y) \ge 1$, Lemma \ref{lem:tp_est_prel} implies $$\tp(1,x,y) \le  CH(x)p_1(y),\quad|x|\leq 2(2\kappa)^{1/\alpha}\leq|y|.$$
The symmetry of $\tp$,  \eqref{eq2:int_tpinfty} with $R=(2\kappa)^{1/\alpha}$, \eqref{eq:oppt} and \eqref{eq:tp_q} imply
\begin{align*}
\tp(1,x,y)& \leq 2p(1,x,y)+c[p_1(y)H(x)\wedge p_1(x)H(y)] \\
& \approx p(1,x,y) +\frac{1}{(|x|+|y|)^{d+\alpha}}\approx p(1,x,y), \qquad |x|,|y| \ge2(2\kappa)^{1/\alpha}.
\end{align*}
The proof is complete.

\end{proof}

\begin{lemma}\label{lem:ESTtp1} 
 For every $\gamma \in (0,\delta)$ there is a constant $C_\gamma$ such that 
\begin{align*}
	\tp(1,x,y) \le  C_\gamma\left(H(x)|y|^{\gamma-d} \land H(y)|x|^{\gamma-d} \right),\quad |x|, |y| \leq 2(2\kappa)^{1/\alpha}. 
\end{align*}
\end{lemma}

\begin{proof}
We will use Lemma \ref{lem:tp_est_prel}. Let $x,y\in\Rd\setminus\{0\}$. By \eqref{eq:1} and Lemma \ref{cor:estintHbis},
\begin{align}
\int_0^1 \int_{\RR^d} \tp(s,x,z) |z|^{-\gamma-\alpha} dz ds 
&\le \frac{1}{\kappa-\kappa_\gamma}   \int_{\RR^d} \tp(1,x,z) |z|^{-\gamma} dz ds 
 \le  \frac{c}{\kappa-\kappa_\gamma}  H(x). \label{eq1:ESTtp1}
\end{align}
 By \eqref{eq:sca},
\begin{align}\label{eq:estp}
	p_s(y) \le C s^{(\beta-d)/\alpha}|y|^{-\beta}, \qquad s>0, \quad 0 \le \beta \le d+\alpha\, ,
\end{align}
in particular, $p_s(y) \le c \left(s^{-d/\alpha} \land \frac{1}{|y|^d} \land \frac{s}{|y|^{d+\alpha}}\right)$. 
This implies 
\begin{align*}
\int^{1/2}_0\int_{2|z|\leq |y|}\tp(1-s,x,z)q(z)p_s(y)dzds &\le c\int_0^1 \int_{|z|\le |y|} \tp(1-s,x,z)|y|^\gamma|z|^{-\gamma-\alpha} |y|^{-d} dz ds \\& \le \frac{c}{\kappa-\kappa_\gamma} |y|^{\gamma-d} H(x).
\end{align*}
We  also have $\sup_{z\in\Rd}p_s(z)\leq cs^{-d/\alpha}\leq c |y|^{-d}$ for $s\geq g(y)$. Furthermore, $p(s,z,y)\approx p_s(z)\leq c|z|^{d}$ for $|z| >2|y|$. Therefore, by \eqref{eq1:ESTtp1}, 
\begin{align*}
&\int^1_{g(y)\wedge 1}\int_{2|z|>|y|} \tp(1-s,x,z)q(z)p(s,z,y)dzds \\
 &\le c|y|^{-d}\int_{g(y)}^1 \int_{|z|<2|y|} \tp(1-s,x,z) q(z) dz ds  +\,c\int_{g(y)}^1 \int_{|z|> 2|y|} \tp(1-s,x,z) q(z) \frac{1}{|z|^{d}} dz ds \\ 
&\le \, c |y|^{\gamma-d}\int_{0}^1 \int_{\RR^d} \tp(1-s,x,z) |z|^{-\gamma-\alpha} dz ds \le  c H(x)|y|^{\gamma-d}.
\end{align*}
Since $p_1(y)\leq c |y|^{\gamma-d}$, we obtain the claim by Lemma \ref{lem:tp_est_prel}.
\end{proof}

\begin{lemma}\label{lem:inttp0} 
For all $\beta \in (\delta, d-\delta)$ and $\eta \in (0, \delta)$ there is a constant $C_\eta$ such that
$$
\int_{\RR^d} \tp(t,x,z) |z|^{-\beta} dz \le C_\eta |x|^{-\beta-\eta}, \qquad 0< t \le 1 , \quad |x| \le 2(2\kappa)^{1/\alpha}.
$$
\end{lemma}
\begin{proof}
Let $\gamma = \delta-\eta$, so that $\gamma \in (0, \delta)$ and $\beta - \gamma >0$. Fix $R=2(2\kappa)^{1/\alpha}$. By Lemma \ref{lem:ESTtp1},
\begin{align*}
\int_{|z| \le R} \tp(1,x,z) |z|^{-\beta} dz &\le  C\int_{|z|\le|x| \le R} |x|^{\gamma-d} |z|^{-\delta-\beta} dz
 + C\int_{|x| < |z| \le R} |x|^{-\delta}  |z|^{\gamma-\beta-d} dz\\ &\le C |x|^{\gamma-\beta-\delta} = C |x|^{-\beta -\eta}.	
\end{align*}
Proposition \ref{lem:int_tpinfty} implies
\begin{align*}
\int_{|z| > R} \tp(1,x,z) |z|^{-\beta} dz \le  R^{-\beta}\int_{\RR^d} \tp(1,x,z) dz \le c H(x) \le c|x|^{-\beta -\eta}.
\end{align*}
This resolves the case of $t=1$. By the scaling of $\tp$ and a change of variables,
\begin{align*}
\int_{\RR^d} \tp(t,x,z) |z|^{-\beta} dz &= \int_{\RR^d} t^{-\beta/\alpha} \tp(1,t^{-1/\alpha}x,w) |w|^{-\beta} dw \\
&\le C t^{-\beta/\alpha} |t^{-1/\alpha}x|^{-\beta-\eta} 
\le C |x|^{-\beta-\eta}.
\end{align*}
\end{proof}

For $\beta >0$ we denote $H_\beta(x) = |x|^{-\beta} + 1$. We notice that $H_\delta=H$. By Lemma \ref{lem:inttp0} and Proposition \ref{lem:int_tpinfty}, for all $\beta \in (0, d-\delta)$ and $\eta \in (0, \delta)$ we have
\begin{align}
	\int_{\RR^d} \tp(t,x,y) H_{\beta}(y) dy \le C_\eta H_{(\beta+\eta)\vee \delta}(x), \qquad t\le1,\;   x\in \RR^d \label{eq:inttp0}.
\end{align}
In the next lemma we improve this result.

\begin{lemma}\label{lem:inttp}
Let $0 <\beta < d-\delta$. There is a constant $C_\beta$ such that 

\begin{align*}
	\int_{\RR^d} \tp(t,x,y) H_\beta(y) dy\le C_\beta t^{-\beta/\alpha}H(t^{-1/\alpha}x), \qquad t>0,\;   x\in \RR^d 
\end{align*}
\end{lemma}
\begin{proof}
By the scaling of $\tp$ it suffices to prove the result for fixed $t>0$.
If $\beta \le \delta$, then we simply apply Lemma \ref{cor:estintHbis} and Proposition \ref{lem:int_tpinfty}, so suppose that $\beta = \delta +\xi \alpha$ for some $\xi >0$. 

First, let $\xi <1$. By Theorem \ref{thm:1} and   Lemma \ref{cor:estintHbis},
\begin{align}
	\int_0^1 \int_{\RR^d} \tp(s,x,z) H_{\gamma}(z) \,dz\,ds &\le c H(x),  \qquad \text{if} \quad \delta < \gamma < \delta+\alpha. \label{eq1:inttp}
\end{align}	
Let $0 < \eps < (1-\xi)\alpha\wedge \delta$, so that $\beta + \eps < \delta+\alpha$. By \eqref{eq:inttp0} and \eqref{eq1:inttp}, 
\begin{align}
\int_{\RR^d} \tp(1,x,y) H_{\beta}(y) dy & = \int_{\RR^d}\int_{\RR^d} \tp(s,x,z)\tp(1-s,z,y) H_{\beta}(y)\, dz\,dy \label{eq1.5:inttp} \\
&=\int_0^1\int_{\RR^d}\int_{\RR^d} \tp(s,x,z)\tp(1-s,z,y) H_{\beta}(y)\, dy\,dz\,ds \notag \\
&\le C \int_0^1\int_{\RR^d} \tp(s,x,z) H_{\beta+\eps}(z) \,dz\,ds  \le C H(x), \notag
\end{align}
as needed.
Next, let $\xi \ge 1$.  By Corollary \ref{cor:tp_bas_est2},
\begin{align}
	\int_0^1\int_{\RR^d} \tp(s,x,z) H_{\gamma}(z) \,dz\,ds &\le c H_{\gamma-\alpha}(x) \quad \text{if} \quad \delta +\alpha < \gamma < d-\delta. \label{eq2:inttp}
\end{align}
Let us fix $\nu \in (0,1)$  such that $\eta := (1-\nu)\alpha < d-\beta-\delta$ and $\eta<\delta$.
By \eqref{eq:inttp0} and \eqref{eq2:inttp},  we have for $\delta+\alpha\leq \gamma\leq \beta$,
\begin{align}
\int_{\RR^d} \tp(1,x,y) H_{\gamma}(y) dy  &=\int_0^1\int_{\RR^d}\int_{\RR^d} \tp(s,x,z)\tp(1-s,z,y) H_{\gamma}(y)\, dy\,dz\,ds  \label{eq3:inttp} \\
&\le C \int_0^1\int_{\RR^d} \tp(s,x,z) H_{\gamma+\eta}(z) \,dz\,ds  \le C H_{\gamma-\nu\alpha}(x). \notag
\end{align}
We choose $n \in \NN$ so that $(n-1)\nu +1 \leq  \xi < n \nu +1$.  By \eqref{eq3:inttp} and \eqref{eq1.5:inttp},
\begin{align*}
\int_{\RR^d} \tp(n+1,x,y) H_{\beta}(y) dy & = \int_{\RR^d}\ldots\int_{\RR^d} \tp(1,x,z_1) \ldots \tp(1,z_n,y) H_{\beta}(y)\, dy\, dz_n\ldots dz_1 \\
& \le C^n\int_{\RR^d} \tp(1,x,z_1) H_{\beta-n\nu\alpha}(z_1)\, dz_1 \\
& = C^n\int_{\RR^d} \tp(1,x,z_1) H_{\delta +(\xi-n\nu)\alpha}(z_1)\, dz_1 \\
&\le C^{n+1} H(x),
\end{align*}
as needed.
\end{proof}

We next improve the estimate from  Lemma \ref{lem:ESTtp1}.
\begin{lemma}\label{lem:ESTtp2} 
Let $0<|x|, |y| \leq 2(2\kappa)^{1/\alpha}$. For each $\eta \in (0,\delta)$ there is  $C_\eta$ such that
\begin{align*}
	\tp(1,x,y) \le C_\eta \left(H(x)|y|^{-\delta-\eta} \land H(y)|x|^{\delta-\eta} \right). 
\end{align*}
\end{lemma}

\begin{proof}
We will use Lemma \ref{lem:tp_est_prel}.
By Lemma \ref{lem:inttp}, for $s \in (0,1/2)$ we have
\begin{align}
\int_{|z| < 2|y|} \tp(1-s,x,z) q(z) dz  &\le \kappa   |2y|^{d-\delta-\eta-\alpha} \int_{\RR^d} \tp(1-s,x,z) |z|^{\delta+\eta-d} dz \notag\\&\le c     |y|^{d-\delta-\eta-\alpha} H(x). \label{eq1:ESTtp2}
\end{align}
Therefore,
\begin{align}\label{eq:101}
\int^{1/2}_0\int_{2|z|\leq |y|}\tp(1-s,x,z)q(z)p_s(y)dzds & \le   c H(x) |y|^{d-\eta -\delta-\alpha} \int_0^\infty p_s(y)\,ds  \\  & \le  C H(x) |y|^{-\delta-\eta}\,.\nonumber
\end{align}
By \eqref{eq:oppt} and \eqref{eq:tp_q},
\begin{equation}\label{eq:102}
\int^1_{1/2}\int_{2|z|>|y|} \tp(1-s,x,z)q(z)p(s,z,y)dzds\leq p_{1/2}(0)MH(x).
\end{equation}
Let $g(y)  < 1/2$. For $|z| > 2|y|$, by \eqref{eq:estp} we have $p(s,z,y) \le s^{-(\delta+\eta+\alpha)/\alpha}/|z|^{d-\delta-\eta-\alpha}$. Hence, by \eqref{eq1:ESTtp2} and Lemma \ref{lem:inttp}, 
\begin{align}
&\int^{1/2}_{g(y)}\int_{2|z|>|y|} \tp(1-s,x,z)q(z)p(s,z,y)dzds  \notag \\
&\le  c\int_{g(y)}^{1/2} \int_{|z|<2|y|} \tp(1-s,x,z) q(z) s^{-d/\alpha} dz ds  \label{eq:103} \\
&+ c\int_{g(y)}^{1/2} \int_{|z|> 2|y|} \tp(1-s,x,z) q(z) \frac{s^{-(\delta+\eta+\alpha)/\alpha}}{|z|^{d-\delta-\eta-\alpha}} dz ds \le C H(x) |y|^{-\delta-\eta}\,.\notag 
\end{align}
Since $p_1(y) \le p_1(0)= c$, combining \eqref{eq:101} -- \eqref{eq:103} with Lemma \ref{lem:tp_est_prel} we see that
$\tp(1,x,y) \le CH(x)|y|^{-\delta-\eta}$. The symmetry of $\tp$ ends the proof.
\end{proof}
Our integral analysis comes to a fruition.
\begin{lemma}\label{lem:ub}
There exists a constant $C$ such that
\begin{align*}
\tp(t,x,y) \le C p(t,x,y) H(t,x)H(t,y), \qquad x,y \in \RR^d \setminus\{0\}.
\end{align*}
\end{lemma}

\begin{proof}
By scaling for $\tp$ (Lemma \ref{ss th}), it suffices to consider $t=1$.
Fix $\eta \in (0,d/2-\delta)$. Let $R=2(2\kappa)^{1/\alpha}$. By \eqref{eq:ChK}, Lemma \ref{lem:ESTtp2} and \ref{lem:ESTtp0},  for $|x|,|y|\leq 2(2\kappa)^{1/\alpha}$ we have 
\begin{align*}
\tp(2,x,y) & = \int_{\RR^d} \tp(1,x,z) \tp(1,z,y)\,dz \\&\le c\int_{|z| \le R} H(x) |z|^{-\delta-\eta} |z|^{-\delta-\eta}  H(y)\,dz + c\int_{|z| > R} H(x) |z|^{-d-\alpha} |z|^{-d-\alpha}  H(y)\,dz \\& \le CH(x)H(y).
\end{align*}
Since $p(1,x,y) \ge c/(1+ |x-y|^{d+\alpha})$, we obtain
\begin{align}
	\tp(1,x,y) \le Cp(1,x,y) H(x) H(y), \qquad |x|,|y| \leq R. \label{eq:tp_est_compact}
\end{align} 
By Lemma \ref{lem:ESTtp0}, we have
\begin{align}
	\tp(1,x,y) \le Cp(1,x,y) H(x) H(y), \qquad  R \le |y|, \label{eq:tp_est_compact1}
\end{align} 
By symmetry, \eqref{eq:tp_est_compact} and \eqref{eq:tp_est_compact1},
\begin{align*}
	\tp(1,x,y) \le C p(1,x,y)H(x)H(y), \qquad x,y \in \RR^d.
\end{align*}
\end{proof}

\subsection{Lower bound of the heat kernel}
For $r>0$ we denote $B_r = B(0,r)$. Let $q(t,x,y)=\tilde p(t,x,y)/(H(x)H(y))$ and $\mu(dy)=H^2(y)dy$. Note that $q$ is an integral kernel of a semigroup on $L^2(\mathbb{R}^d,\mu)$. By \eqref{eq:2} and Corollary \ref{cor:sh}, 
\begin{align*}
1\leq \int_{\Rd} q(1,x,y)\mu(dy)\leq M+1,\quad x\neq 0.
\end{align*}
Next, by Lemma \ref{lem:ub},
$$q(1,x,y)\leq C_1p(1,x,y),\quad x,y\neq 0.$$
Therefore, by \eqref{eq:oppt}, there is $R>2$ such that 
for $x\in B_1\setminus\{0\}$,
$$\int_{B_R^c}q(1,x,y)\mu(dy)\leq \frac14.$$
Furthermore, there is $0<r<1$ such that for all $x\neq 0$,
$$\int_{B_r}q(1,x,y)\mu(dy)\leq C_1p_1(0)\mu(B_r)\leq \frac 14.$$
Hence, for $x\in B_1\setminus\{0\}$,
\begin{align}
	\int_{B_R\setminus B_r}q(1,x,y)\mu(dy)\geq \frac12. \label{eq0:tp_lb}
\end{align}
By the semigroup property and \eqref{eq0:tp_lb}, for $x\in B_1\setminus\{0\}$ and $|y|\geq r$ we have
\begin{align}
q(2,x,y) &\geq \int_{B_R\setminus B_r}q(1,x,z)q(1,z,y)\mu(dz) \geq \int_{B_R\setminus B_r}q(1,x,z)\frac{p(1,z,y)}{H^2(r)}\mu(dz) \notag \\ 
&\geq c_2 \frac{(R+|y|)^{-d-\alpha})}{H^2(r)}\int_{B_R\setminus B_r}q(1,x,z)\mu(dz)
\geq c_3 p_1(y). \label{eq1:tp_lb}
\end{align}
In a similar way we get
\begin{align}
q(3,x,y) \geq c_4 p_1(y) \ge c_5 p(3,x,y), \qquad x\in B_1\setminus\{0\}, \; |y|\geq 1\,. \label{eq2:tp_lb}
\end{align}
Now, let $x,y\in B_1\setminus\{0\}$. Again by the semigroup property and \eqref{eq1:tp_lb}, 
\begin{align}q(3,x,y)& \geq \int_{B_R\setminus B_r}q(1,x,z)q(2,z,y)\mu(dz) \geq c_3 \int_{B_R\setminus B_r}q(1,x,z)p_1(z)\mu(dz) \notag \\
&\geq c_6 p_1(y)\int_{B_R\setminus B_r}q(1,x,z)\mu(dz) \geq \frac{c_6}{2} p_1(y) \ge c_7 p(3,x,y). \label{eq3:tp_lb}
\end{align}
If $|x|,|y| >1$, then $q(3,x,y) \ge H^2(1) \tp(3,x,y) \ge H^2(1) p(3,x,y)$. By \eqref{eq2:tp_lb}, \eqref{eq3:tp_lb} and symmetry, 
$$q(3,x,y)\geq c_8 p(3,x,y),\quad x,y\neq 0.$$
Thus,
$$\tilde p(3,x,y)\geq C p(3,x,y)H(x)H(y),\quad x\neq 0.$$
By the scaling of $\tp$ we get
\begin{equation}\label{eq:lb}
\tilde p(t,x,y)\geq C p(t,x,y)H(t,x)H(t,y),\quad x\neq 0.
\end{equation}

\subsection{Continuity of the heat kernel}

\begin{lemma}\label{lem:tp_cont1}
For every $x \in  \RR^d \setminus \{0\}$, the function $\RR^d \setminus \{0\}\ni y\mapsto \tp(1,x,\cdot)$ is continuous. 
\end{lemma}
\begin{proof}
Fix $x,y \not=0$ and let $w \to y$. Then,
\begin{align*}
\tp(1,x,y) - \tp(1,x,w) &= p(1,x,y) - p(1,x,w) \\
&+ \int_0^1 \int_{\RR^d} \tp(1-s,x,z) q(z) [p(s,z,y) - p(s, z,w)]\,dz\,ds. 
\end{align*}
The  integral converges to 0. Indeed,  let $\eps >0$ be small. 
By Bogdan and Jakubowski \cite[Theorem 4]{MR2283957},
$$ p(1-s,x,z) p(s,z,y)\leq c p(1,x,y)[p(1-s,x,z) + p(s,z,y)],\quad 0<s<1,\, x,y,z\in\Rd.$$
Therefore,
\begin{align}
&\int_0^\eps \int_{\RR^d} \tp(1-s,x,z) q(z) p(s,z,y)\,dz\,ds  \label{eq:int_eps}\\
&\le c \int_0^\eps \int_{\RR^d} p(1-s,x,z) H(x)H(z) q(z) p(s,z,y)\,dz\,ds \notag\\
& \le    c_1 p(1,x,y) H(x) \int_0^\eps \int_{\RR^d} [p(1-s,x,z) + p(s,z,y)] H(z)q(z) \,dz\,ds \notag\\
&\le c_2 \eps H(x)[q(x)H(x)+q(y)H(y)] \,. \notag
\end{align}
For $s \in (\eps,1)$ we have $p(s,z,w) \approx p(s,z,y)$. By the dominated convergence theorem, 
\begin{align*}
\int_\eps^1 \int_{\RR^d} \tp(1-s,x,z) q(z) [p(s,z,y) - p(s, z,w)]\,dz\,ds \to 0 \quad \mbox{ as } \quad w \to y.
\end{align*}
This completes the proof.
\end{proof}

\begin{lemma}\label{lem:c}
The function $\tp(t,x,y)$ is jointly continuous in $t >0$ and $x,y \in \RR^d \setminus \{0\}$. 
\end{lemma}
\begin{proof}
By scaling, it suffices to show the continuity of  $\tp(1,\cdot,\cdot)$. Fix $x,y \in \RR^d \setminus \{0\}$ and let $\tilde{x} \to x$ and $\tilde{y} \to y$.
As in Lemma \ref{lem:tp_cont1}, we only need to show that
\begin{align*}
\int_0^1 \int_{\RR^d} |\tp(1-s,\tilde{x},z) p(s,z,\tilde{y}) - \tp(1-s, x,z)p(s,z,y)| q(z)\,dz\,ds \rightarrow 0.
\end{align*}
In addition to \eqref{eq:int_eps} we get
\begin{align}
&\int_{1-\eps}^1 \int_{\RR^d} \tp(1-s,x,z) q(z) p(s,z,y)\,dz\,ds \label{eq:int_1-eps}\\
&\le c \int_0^\eps \int_{\RR^d} p(s,x,z) H(s,x)H(s,z) q(z) p(1-s,z,y)\,dz\,ds \notag\\
& \le    c_1 p(1,x,y) \int_0^\eps \int_{\RR^d} [p(s,x,z) + p(1-s,z,y)] H(x)H(z)q(z) \,dz\,ds \notag\\
&\le c_2 \eps H(x)[q(x)H(x)+q(y)H(y)] \,. \notag
\end{align}
By \eqref{eq:int_eps} and \eqref{eq:int_1-eps}, 
\begin{align*}
&\int_0^1 \int_{\RR^d} |\tp(1-s,\tilde{x},z) p(s,z,\tilde{y}) - \tp(1-s, x,z)p(s,z,y)| q(z)\,dz\,ds \\
& \le \int_0^\eps\int_{\RR^d} [\tp(1-s,\tilde{x},z) p(s,z,\tilde{y}) + \tp(1-s, x,z)p(s,z,y)] q(z)\,dz\,ds \\
& + \int_{1-\eps}^1 \int_{\RR^d} [\tp(1-s,\tilde{x},z) p(s,z,\tilde{y}) + \tp(1-s, x,z)p(s,z,y)] q(z)\,dz\,ds \\
& + \int_\eps^{1-\eps} \int_{\RR^d} |\tp(1-s,\tilde{x},z) p(s,z,\tilde{y}) - \tp(1-s, x,z)p(s,z,y)| q(z)\,dz\,ds \\
& \le \eps c(x,y) + \int_\eps^{1-\eps} \int_{\RR^d} |p(1-s,\tilde{x},z) p(s,z,\tilde{y}) - p(1-s, x,z)p(s,z,y)| q(z)\,dz\,ds,
\end{align*}
where $c(x,y)$ is a constant.
By Lemma \ref{lem:ub} and \eqref{eq:lb}, for $\eps < s < 1-\eps$ we have $\tp(1-s,\tilde{x},z) \approx \tp(1-s,x,z)$ and $p(s,z,\tilde{y}) \approx p(s,z,y)$. By Lemma \ref{lem:tp_cont1} and the dominated convergence theorem, the last integral  is arbitrarily small. This completes the proof.
\end{proof}

\subsection{Proof of Theorem \ref{thm:main}}
We combine Lemma~ \ref{lem:ub},
inequality \eqref{eq:lb} and Lemma~\ref{lem:c}.

\subsection{Blowup} 
\begin{corollary}
If $\kappa > \kappa^*$ in \eqref{eq:SchrOp}, then $\tp(t,x,y) \equiv \infty$.
\end{corollary}
\begin{proof}
Clearly, $\tp \ge \tp^*$. By Theorem \ref{thm:main},
\begin{align}\label{eq:psar}
\tp^*(t,x,y) \ge c p(t,x,y) t^{d-\alpha}|x|^{-(d-\alpha)/2}|y|^{-(d-\alpha)/2}\,.
\end{align}
The kernel $\tp$ may be considered as a perturbation of $\tp^*$ by $(\kappa-\kappa^*)|x|^{-\alpha}$ \cite[Lemma~8]{MR2457489}. By Duhamel's formula and \eqref{eq:psar},
\begin{align*}
\tp(t,x,y) &= \tp^*(t,x,y) + (\kappa-\kappa^*)\int_0^t \int_{\Rd} \tp(t-s,x,z) |z|^{-\alpha} \tp^*(s,x,z)\,dz\,ds \\
&\ge (\kappa-\kappa^*)\int_0^t \int_{\Rd} \tp^*(t-s,x,z) |z|^{-\alpha} \tp^*(s,x,z)\,dz\,ds = \infty.
\end{align*}
\end{proof}

\section{Quadratic forms}\label{sec:P}

In this section we give a different proof of the upper bound for $\tp$ in Theorem \ref{thm:main}.
To this end we employ Davies' method \cite{MR882426} with the usual setting of quadratic forms and Sobolev inequalities, as extended to nonlocal operators by Carlen, Kusuoka and Stroock \cite{MR898496}, see also Grigor'yan \cite[Section 5 and 10]{MR2218016} and Murat and Saloff-Coste \cite[Introduction]{MR3601569} for recent perspectives. 
Recall that Davies' method consists in Doob-type conditioning by a suitable nonnegative function and deriving the on-diagonal estimates for the conditioned semigroup from appropriate Sobolev inequalities. The argumnt depends on a careful analysis of the domain of the quadratic form of $\tp$.

We first define the quadratic form ${\mathcal E}$ of $\Delta^{\alpha/2}$, cf. \eqref{eq:Hi}, in the usual way \cite{MR2778606}: 
$$
{\mathcal E}
[f]=\lim_{t\to 0+}\frac{1}{t}(f-P_t f,f)
=\frac12 \int_\Rd \int_\Rd [f(x)-f(y)]^2\nu(y-x)dxdy,\qquad f\in L^2(\Rd).$$
We let $$\mathcal{D}(\mathcal{E})=\{f\in L^2(\Rd):\mathcal{E}[f]<\infty\}.$$
Similarly, although less explicitly we define 
$$
\tilde {\mathcal E}
[f]=\lim_{t\to 0+}\frac{1}{t}(f-\tilde P_t f,f),
\qquad f\in L^2(\Rd),$$
and $$\mathcal{D}(\tilde{\mathcal{E}})=\{f\in L^2(\Rd):\tilde{\mathcal{E}}[f]<\infty\}.$$
This form is pivotal. To handle it we 
introduce an auxiliary quadratic expression:
\begin{equation}\label{eq:Ezk}
\overline{\mathcal{E}}[f]=\frac12\int_{\R^d}\!\int_{\R^d} 
\left[\frac{f(x)}{h(x)}-\frac{f(y)}{h(y)}\right]^2
h(x)h(y) \nu(y-x)
\,dy\,dx,\quad f\in L^2(\Rd),
\end{equation}
and we let $\mathcal{D}(\overline{\mathcal{E}})=\{f\in L^2(\Rd):\overline{\mathcal{E}}[f]<\infty\}$.

\begin{lemma}\label{lem:form_domain}
We have $\mathcal{D}(\mathcal{E})\subset \mathcal{D}(\tilde{\mathcal{E}})\subset \mathcal{D}(\overline{\mathcal{E}})$ and 
$$\tilde{\mathcal{E}}[f]=\overline{\mathcal{E}}[f]=\mathcal{E}[f]-\int_{\Rd}f^2(x)q(x)dx,\quad f\in \mathcal{D}(\mathcal{E}).$$
\end{lemma}
\begin{proof}
 As usual, we denote $(f,g) = \int_{\Rd} f(x) g(x) d x$ for $f,g \in L^2(\Rd)$. 
Recall that $\{P_t\}_{t\geq 0}$ is a strongly continuous semigroup in $L^2(\Rd)$.
  Also, each $P_t$ is symmetric on $L^2(\Rd)$:
\begin{equation*}
   (P_t f,g) = (f,P_tg),
\end{equation*}
where the application of the Fubini's theorem is justified because
\begin{equation*}
(P_t |f|, |g|) \leq \|f\|_{L^2}\|g\|_{L^2}<\infty. 
\end{equation*}
For $f\in L^2(\Rd)$ we let
\begin{equation*}
  \mathcal{E}_t [f] = \tfrac{1}{t} (f-P_t f, f), \quad t>0, 
\end{equation*}
and 
\begin{equation*}
  \mathcal{E} [f] = \sup_t \mathcal{E}_t [f].
\end{equation*}
We have 
\begin{align*}
  \mathcal{E}_t [f] &= \tfrac{1}{t} (f,f) -\tfrac{1}{t}(P_t f,f)  = \tfrac{1}{t}(f,f) -\tfrac{1}{t}(P_{t/2}f, P_{t/2}f)\\
  &= \tfrac{1}{t}\left[ (f,f) - (P_{t/2}f, P_{t/2}f )\right] \geq 0,
\end{align*}
because $P_{t/2}$ is a contraction. By the spectral theorem, 
$(0,\infty)\ni t\mapsto \mathcal{E}_t[f]$ is non-increasing function, hence $\mathcal{E}_t[f] \to \mathcal{E}[f]$ as $t\to 0$ \cite[Lemma~1.3.4]{MR2778606}.
By Proposition \ref{prop:SCP}, $\tilde P_t$ is also a contraction. Let $f\in L^2(\Rd)$
and $\tilde {\mathcal{E}}[f]=\sup_t  \tilde {\mathcal{E}} _t [f]$, where $t>0$ and
\begin{equation*}
  \tilde {\mathcal{E}} _t [f] = \tfrac{1}{t} (f -\tilde P_t f,f) =  \tfrac{1}{t}\left[ (f,f) - (\tilde P_{t/2}f, \tilde P_{t/2}f )\right].
\end{equation*}

By Theorem \ref{thm:1} we have $\tilde P_t h= h$. For $f\in L^2(\Rd)$ we get
\begin{align*}
f(x) - \tilde{P_t}f(x) 
&=\int_{\RR^d}\tilde{p}(t,x,y)\left(\frac{f(x)}{h(x)}-\frac{f(y)}{h(y)}\right)h(y)dy,
\end{align*}
hence,
\begin{align*}
(f- \tilde{P}_t f,f) =\int_{\RR^d}\int_{\RR^d}\tilde{p}(t,x,y)\left(\frac{f(x)}{h(x)}-\frac{f(y)}{h(y)}\right)\frac{f(x)}{h(x)}h(x)h(y)dxdy.
\end{align*}
By the symmetry of $\tp$ we also have
\begin{align*}
(f- \tilde{P}_tf,f) = \int_{\RR^d}\int_{\RR^d}\tilde{p}(t,x,y)\left(\frac{f(y)}{h(y)}-\frac{f(x)}{h(x)}\right)\frac{f(y)}{h(y)}h(x)h(y)dxdy.
\end{align*}
Therefore,
\begin{align*}
  \tilde{\mathcal{E}}_t[f] &= \tfrac{1}{t}(f-\tilde P_t f,f) \\
  & = \tfrac{1}{2t} \int_{\bbfR^d}\int_{\bbfR^d} \tilde p(t,x, y) \left( \frac{f(x)}{h(x)} - \frac{f(y)}{h(y)}\right)^2 h(x) h(y) d x d y \\
  &\geq \tfrac{1}{2} \int_{\bbfR^d}\int_{\bbfR^d}\frac1t p(t,x,y) \left(\frac{f(x)}{h(x)}-\frac{f(y)}{h(y)}\right)^2 h(x)h(y) d x d y.
\end{align*}
We note that $\frac1t p(t,x,y)\to \nu(y-x)$ as $t\to 0$, which follows, e.g., from  \cite[Theorem 2.1]{MR0119247} and scaling.
By Fatou's lemma, 
\begin{equation*}\tilde{\mathcal{E}}[f]\geq \overline{\mathcal{E}}[f],\quad f\in L^2(\Rd).
\end{equation*}
Therefore $ \mathcal{D}(\tilde{\mathcal{E}})\subset \mathcal{D}(\overline{\mathcal{E}})$. By   \cite[Proposition~5]{MR3460023}\footnote{The trivial factor $\tfrac{1}{2}$ is missing in \cite[Proposition~5]{MR3460023}.},
\begin{align}\label{eq:1.5}
	\tEe[f] + \int_{\Rd} f^2(x) q(x)\,dx \ge \overline{\mathcal{E}}[f]+\int_{\Rd} f^2(z) q(z) d z =  \mathcal{E} [f],\quad f\in L^2(\Rd).
\end{align}
We next consider an approximation of $\tilde{\mathcal E}$. Let $\ve >0$ and $0<q^{\ve } := q \wedge \tfrac{1}{\ve}\le 1/\ve$. Let $\tilde p^{\ve}$  be the Schr\"odinger perturbation of $p$ by $q^{\ve}$, i.e. we let
\begin{align*}  
p^\ve_0(t,x,y) &= p(t,x, y),\\
p^\ve_{n}(t,x,y) & = \int_0^t \int_{\Rd}p^\ve_{n-1}(s,x,z) q^{\ve} (z) p(t-s,z, y) d z d s, \quad n=1,2,\ldots,\\
\tilde p^{\ve} (t,x,y) &= \sum_{n=0}^{\infty }p^{\ve}_n(t,x,y).
\end{align*}
By induction,
\begin{equation*}
  0\le p^\ve_n(t,x,y) \leq p(t,x, y) (t/\ve)^n/n!,
\end{equation*} 
hence, uniformly in $x, y\in \Rd$,
\begin{align*}
  0\le p^\ast(t,x,y) &:= \tilde p^{\ve}(t,x,y) - p(t,x, y) - p^\ve_1(t,x,y) \\
  &\leq p(t,x, y) \left( e^{\frac{t}{\ve}} - 1 - \tfrac{t}{\ve}\right) = p(t,x, y) O(t^2) \quad \text{as}\quad t\to 0.
\end{align*}
We denote 
$P^\ast_t f(x) = \int_{\Rd}p^*(t,x,y) f(y) d y$ and we get
\begin{equation*}
  \tfrac{1}{t} \left|(P_t^\ast f,f)\right| \leq c \tfrac{1}{t} t^2 \|f\|_{L^2}^2 \to 0 \quad \text{as}\quad t \to 0.
\end{equation*}
For $f \in L^2(\Rd)$, we  let
$
  P^{(1,\ve)}_t f(x) = \int_{\Rd}p^\ve_1(t,x, y) f(y) d y,
$
and we obtain
\begin{align*}
  \tfrac{1}{t}(P^{(1,\ve)}_t f,f) & = \tfrac{1}{t}\int_{\Rd}\int_{\Rd}\int_0^t \int_{\Rd} p(s,x,z)q^{\ve}(z)p(t-s,z,y)f(y)f(x) d z d s d y d x\\
&= \tfrac{1}{t} \int_{\Rd} \int_0^t P_sf(z) q^{\ve}(z) P_{t-s}f(z) d s d z. 
\end{align*}
Since $\{P_t\}_{t\geq 0}$ is strongly continuous semigroup of contraction on  $L^2(\Rd)$,  the above converges to $\int_{\Rd}f^2(z) q^{\ve} (z) d z$ as $t\to 0$.
 We have that $\tilde P^\ve_t f(x)= \int_{\Rd}\tilde p^\ve(t,x, y) f(y) d y$ is a symmetric strongly continuous semigroup on $L^2(\Rd)$. Each $\tilde P_t^\ve$ is contractive because it is smaller than $\tilde P_t$ on nonnegative functions.
 
We note that
$\tilde {\mathcal{E}}_t^{\ve}[f]:= \tfrac{1}{t} (f -\tilde P^\ve_t f,f)\to \tilde {\mathcal{E}}^{\ve}[f]:=\sup_s \tilde {\mathcal{E}}^{\ve}_s[f]$ as $t\to 0$. Therefore,
\begin{equation*}
   \tilde{\mathcal{E}}^{\ve} [f] = \mathcal{E}[f] - \int_{\Rd}f^2(z) q^{\ve} (z) d z, \quad f\in L^2(\Rd).
\end{equation*}  
If $L^2(\Rd)\ni f\geq 0$, then $\tilde P_t f\geq \tilde P^\ve_t f$, so $\tilde {\mathcal{E}}_t [f] \leq \tilde {\mathcal{E}}^\ve_t [f]$, and
\begin{equation*}
  \tilde{\mathcal{E}}[f]  +\int_{\Rd} f^2(z) q(z) d z \leq \mathcal{E} [f].
\end{equation*}
By \eqref{eq:1.5}, we conclude that for nonnegative $f\in L^2(\R^d)$, 
\begin{equation*}
\tilde{\mathcal{E}}[f]+\int_{\Rd} f^2(z) q(z) d z = \mathcal{E} [f] = \overline{\mathcal{E}}[f]+\int_{\Rd} f^2(z) q(z) d z.
\end{equation*}
For general signed $f\in\mathcal{D}(\mathcal{E})$ we have $f_+,\,f_-\in \mathcal{D}(\mathcal{E})$. Since $\mathcal{D}(\mathcal{E})$, $\mathcal{D}(\tilde{\mathcal{E}})$ are  linear spaces, we obtain that $\mathcal{D}(\mathcal{E})\subset \mathcal{D}(\tilde{\mathcal{E}})$.
We	 will extend \eqref{eq:2} to $f$ with variable sign. Recall that by the general theory \cite[Chapter~1]{MR2778606},
\begin{equation*}
  \tilde{\mathcal{E}}(f,g) = \lim_{t\to 0^+} \tfrac{1}{t} (f - \tilde P_t f,g)
\end{equation*}
exists for general signed $f,g\in \mathcal{D}(\tilde{\mathcal{E}})$, i.e.  if $  \tilde{\mathcal{E}}[f]<\infty$ and $\tilde{\mathcal{E}}[g]<\infty$. Furthermore,
\begin{equation*}
  \tilde{\mathcal{E}}(f,g) = \tfrac{1}{2}\left( \tilde{\mathcal{E}}[f+g] - \tilde{\mathcal{E}}[f] - \tilde{\mathcal{E}}[g]\right).
\end{equation*}
Let $f,g \geq 0$ and $f,g\in\mathcal{D}(\mathcal{E})$. By Hardy inequality and \rf{eq:2},
\begin{align*}
   \tilde{\mathcal{E}}(f,g) &= \tfrac{1}{2}\Bigg[ \mathcal{E}[f+g] - \int_{\bbfR^d}(f(x)+g(x))^2 q (x) d x - \mathcal{E}[f] \\
&+ \int_{\bbfR^d}f^2(x) q(x) d x - \mathcal{E}[g] + \int_{\bbfR^d} g^2(x) q(x)d x \Bigg] \\
   & = \mathcal{E}(f,g) - \int_{\bbfR^d} f(x)g(x) q(x) d x.
\end{align*}
Since $\mathcal{D}(\mathcal{E})\subset \mathcal{D}(\tilde{\mathcal{E}})$, it follows that for arbitrary (signed) $f\in \mathcal{D}(\mathcal{E})$, 
\begin{align*}
   \tilde{\mathcal{E}}[f] &= \tilde{\mathcal{E}}[f_+] + \tilde{\mathcal{E}}[f_-] - 2 \tilde{\mathcal{E}}(f_+,f_-) \\
   &= \mathcal{E}[f_+] - \int_{\bbfR^d} f_+^2(x) q(x) d x + \mathcal{E}[f_-] - \int_{\bbfR^d} f_-^2(x) q(x) d x - 2\mathcal{E}(f_+,f_-) \\
&+ 2 \int_{\bbfR^d} f_+(x)f_-(x) q(x) d x= \mathcal{E}(f_+-f_-, f_+ -f_-) - \int_{\bbfR^d} (f_+(x) -f_-(x))^2 q(x) d x\\
   &= \mathcal{E}[f] - \int_{\bbfR^d} f^2(x) q(x) d x =\overline{\mathcal{E}}[f],
\end{align*}
where in the last line we used \eqref{eq:1.5}
\end{proof}

We will be concerned with approximating $\mathcal{D}(\overline{\mathcal{E}})$ by smooth functions. Let $\rho\in C^\infty_c(\Rd)$ be a radial and radially nonincreasing nonnegative function such that $\mathrm{supp}\,\rho =
\overline{B(0,1)}$ and $\int_{\Rd}\rho(x) dx = 1$. For $\varepsilon > 0$, set $\rho_\varepsilon(x)=\varepsilon^{-d}\rho(x/\varepsilon)$.
\begin{lemma}\label{lem:den1}There exists $c=c(d,\delta)$ such that  
$$\overline{\mathcal{E}}[h(\rho_\varepsilon*f)]\leq c\rho(0)\, \overline{\mathcal{E}}[hf],\quad f\in L^1_{loc}(\Rd).$$
\end{lemma}
\begin{proof}
By Jensen's inequality, 
\begin{align*}
2\overline{\mathcal{E}}[h(\rho_\varepsilon*f)]&=\int\int\left(\int(f(x+z)-f(y+z))\rho_\varepsilon(z)dz\right)^2\nu(y-x)h(x)h(y)dxdy\\
&\leq \int \int\int(f(x+z)-f(y+z))^2\nu(y-x)h(x)h(y)dxdy\rho_\varepsilon(z)dz\\
&= \int \int(f(x)-f(y))^2\nu(y-x)\int h(x-z)h(y-z)\rho_\varepsilon(z)dzdxdy.
\end{align*}
Let 
$$
\mathrm{I}(x,y)= \int h(x-z)h(y-z)\rho_\varepsilon(z)dz.
$$
Note that $\supp \rho_\eps = \overline{B(0,\eps)}$ and $\rho_\eps (z) \le \eps^{-d} \rho(0)$. Denote $\omega_d = |B(0,1)|$.
Let $|x|\leq |y|$. If $|x|\geq 2\varepsilon $, then $\mathrm{I}(x,y)\leq 2^{2\delta} \omega_d \rho(0) h(x)h(y) $. For $|x|<2\varepsilon\leq |y|$, by the monotonicity of $h$ and $\rho$ we get
$$\mathrm{I}(x,y)\leq  2^{\delta}h(y)\varepsilon^{-d}\rho(0)\int_{B(x,3\varepsilon)}h(x-z)dz  =  2^{\delta}3^{d-\delta}\frac{\omega_d}{d-\delta} h(y)\varepsilon^{-\delta}\rho(0)\leq c\rho(0)h(y)h(x),$$
where $c= 2^{2\delta}3^{d-\delta}\omega_d/(d-\delta)$. If $|y|<2\varepsilon$, then by the rearrangement inequality \cite[Theorem 3.4]{MR1817225},
$$\mathrm{I}(x,y)\leq \varepsilon^{-d}\rho(0) \int_{B(0,3\varepsilon)}h^2(z)dz = \rho(0)3^{d-2\delta}\frac{\omega_d}{d-2\delta}  \varepsilon^{-2\delta} \leq c\rho(0)h(x)h(y),$$
where $c=2^{2\delta}3^{d-2\delta}\omega_d/(d-2\delta)$. Hence,
\begin{equation}\label{I_est}
\mathrm{I}(x,y)\leq c\rho(0) h(x)h(y),\quad x,y\in \Rd,
\end{equation}
where $c=c(d,\delta)$. The proof is complete.
\end{proof}

\begin{lemma}\label{lem:den2}Let $\chi\in C^\infty(\Rd)$ and $\textbf{1}_{B^c_{1}}\leq\chi\leq\textbf{1}_{B^c_{1/2}}$.
For all $f\in C^\infty_c(\Rd)$ and $\varepsilon\in(0,1)$,
$$\overline{\mathcal{E}}[h(f\chi(\cdot/\varepsilon))]\leq 2\overline{\mathcal{E}}[hf] + 
c(d,\alpha)\left[\left(\|\nabla \chi\|_\infty\|f\|_\infty+\|\nabla f\|_\infty\right)^2 + \|f\|_\infty^2\right].$$
\end{lemma}
\begin{proof}Denote $f_\eps(x) = f(x) \chi(x/\eps)$ and $A(x,y)=[f_\eps(x)-f_\eps(y)]^2\nu(y-x)h(x)h(x)$.
Since $f_\eps(x)=f(x)$ for $x\in B^c_{\varepsilon}$, 
$$\iint_{B^c_{\varepsilon}\times B^c_{\varepsilon}}A(x,y)dxdy\leq 2\overline{\mathcal{E}}[hf].$$
Observe that
$$|f_\eps(x)-f_\eps(y)|\leq |x-y|(\|\nabla f\|_\infty+ \varepsilon^{-1}\|\nabla\chi\|_\infty\|f\|_\infty).$$
Thus,
\begin{align*}\iint_{B_{2\varepsilon}\times B_{2\varepsilon}}A(x,y)dxdy&\leq \varepsilon^{-2}(\|\nabla f\|_\infty+\|\nabla\chi\|_\infty\|f\|_\infty)^2\iint_{B_{2\varepsilon}\times B_{2\varepsilon}}|x-y|^2\nu(y-x)h(x)h(x)dxdy\\
&=c(d,\alpha)\varepsilon^{d-\alpha}h(\varepsilon)^2(\|\nabla f\|_\infty+\|\nabla\chi\|_\infty\|f\|_\infty)^2\\
&\leq c(d,\alpha)(\|\nabla f\|_\infty+\|\nabla\chi\|_\infty\|f\|_\infty)^2.
\end{align*}
Furthermore,
\begin{align*}\iint_{B_{\varepsilon}\times B^c_{2\varepsilon}}A(x,y)dxdy&\leq 4\|f\|_\infty^2\iint_{B_{\varepsilon}\times B^c_{2\varepsilon}}\nu(y-x)h(x)h(y)dxdy\\
&=c(d,\alpha)\varepsilon^{d-\alpha}h(\varepsilon)^2\|f\|_\infty^2\leq c(d,\alpha)\|f\|_\infty^2,
\end{align*}
which ends the proof.
\end{proof}

\begin{theorem}\label{prop:Hilbert}
$(\mathcal{D}(\tilde{\mathcal{E}}),\sqrt{\tilde{\mathcal{E}}} + \|\cdot\|_{L^2})$ is a Hilbert space and $C^\infty_c(\R^d)$ is dense in $\mathcal{D}(\tilde{\mathcal{E}})$ with the norm $\sqrt{\tilde{\mathcal{E}}} + \|\cdot\|_{L^2}$ . Furthermore, $\tilde{\mathcal{E}}=\overline{\mathcal{E}}$ on $L^2(\Rd)$.
\end{theorem}
\begin{proof}
 By Proposition \ref{prop:SCP} and  \cite[Lemma 1.3.2 and 1.3.4 and Theorem 1.3.1]{MR2778606} we obtain that
$(\tilde{\mathcal{E}},\mathcal{D}(\tilde{\mathcal{E}}))$ is closed. This means that $(\mathcal{D}(\tilde{\mathcal{E}}),\sqrt{\tilde{\mathcal{E}}} + \|\cdot\|_{L^2})$ is a Hilbert space.

We next consider $\overline{\mathcal{E}}$ defined by \eqref{eq:Ezk}.
Let $f_n$ be a Cauchy sequence in $(\mathcal{D}(\overline{\mathcal{E}}),\sqrt{\overline{\mathcal{E}}} + \|\cdot\|_{L^2})$. It is bounded in this norm and it is  a Cauchy sequence in $L^2(\Rd)$, hence it converges to $f$ in $L^2(\Rd)$.
Furthermore, there exists a subsequence $f_{n_k}$ such that it converges pointwise to $f$. Using Fatou's lemma we get
$$\overline{\mathcal{E}}[f]\leq \liminf_{k\to \infty}\overline{\mathcal{E}}[f_{n_k}]<\infty,$$
hence $f\in \mathcal{D}(\overline{\mathcal{E}})$. Since $f_n$ is a Cauchy sequence, by Fatou's lemma,
\begin{equation*}
  \lim_{n\to\infty}\overline{\mathcal{E}}[f-f_n]\leq \lim_{n\to\infty}\liminf_{k\to \infty}\overline{\mathcal{E}}[f_{n_k}-f_{n}]=0,
\end{equation*}
so $f_n$ converges to $f$ in $(\mathcal{D}(\overline{\mathcal{E}}),\sqrt{\overline{\mathcal{E}}} + \|\cdot\|_{L^2})$. Thus, $(\mathcal{D}(\overline{\mathcal{E}}),\sqrt{\overline{\mathcal{E}}} + \|\cdot\|_{L^2})$ is a Hilbert space.

We next prove that $C^\infty_c(\Rd)$ is dense in  $\mathcal{D}(\overline{\mathcal{E}})$ with the norm $\sqrt{\overline{\mathcal{E}}} + \|\cdot\|_{L^2}$. 
Let
$$\mathcal{Q}[f]=\overline{\mathcal{E}}[hf], \; f\in L^2(\Rd,h^2(x)dx)\quad \text{ and } \quad
\mathcal{D}(\mathcal{Q})=\frac1h\mathcal{D}(\overline{\mathcal{E}})=\{f/h: f\in \mathcal{D}(\overline{\mathcal{E}})\}.$$
Thus, $\mathcal{D}(\overline{\mathcal{E}})=h\mathcal{D}(\mathcal{Q})$.
Let $f\in \mathcal{D}(\mathcal{Q})$.
By Lemma~\ref{lem:den1}
the sequence $f_n:=\rho_{1/n}*f$ is bounded in the norm $\sqrt{\mathcal{Q}} + \|\cdot\|_{L^2(h^2(x)dx)}$. Indeed, by Jensen's inequality, Fubini's theorem and \rf{I_est}, 
\begin{align*}
\|f_n\|^2_{L^2(h^2(x)dx)}&=\|\rho_{1/n}*f\|^2_{L^2(h^2(x)dx)} \leq \int_{\Rd}\int_{\Rd}\rho_{1/n}(x-y)f^2(y)h^2(x) d x d y \\
&\leq c\rho(0)\int_{\Rd}f^2(y) h^2(y) d y =  c\rho(0)\|f\|^2_{L^2(h^2(x)dx)}.
\end{align*}
Since $\mathcal{D}(\mathcal{Q})$ is a Hilbert space, by the Banach-Saks theorem there exists a subsequence $\{f_{n_k}\}$ and $g\in \mathcal{D}(\mathcal{Q})$ such that $$\frac{1}{k}\sum^k_{m=1}f_{n_m}\to g.$$ 
Since $f_n\to f$ in $L^2(h^2(x)dx)$, $g=f$. Thus, $C^\infty(\Rd)\cap \mathcal{D}(\mathcal{Q})$ is dense in $ \mathcal{D}(\mathcal{Q})$. 

Now, fix $f\in C^\infty(\Rd)\cap \mathcal{D}(\mathcal{Q})$. Let $\chi_1\in C^\infty_c(\Rd)$ be such that $\textbf{1}_{B_1}\leq\chi_1\leq \textbf{1}_{B_2}$. Denote $f_1=f\chi _{1}$ and $f_2=f-f_1$. Then, $f=f\chi _{1} + f(1-\chi _{1})=f_1+f_2$.
 By the proof of Lemma \ref{lem:den2}, $C^\infty_c(\Rd)\subset \mathcal{D}(\mathcal{Q})$. Hence $f_1\in \mathcal{D}(\mathcal{Q})$ and so $f_2 \in \mathcal{D}(\mathcal{Q})$. Since $g=hf_2 \in \mathcal{D}(\overline{\mathcal{E}})$ and 
\begin{equation*}
  \int_{\Rd} g^2(x) q(x) d x \leq \kappa\int_{B_1^c}f^2(x)h^2(x) d x <\infty,
\end{equation*}
cf. \eqref{eq:dk}, we get $g\in \mathcal{D}(\mathcal{E})$ by \cite[Proposition~5]{MR3460023}. Since $\mathcal{E}$ is a regular Dirichlet form (see, e.g., \cite[Theorem 2.4]{MR2888033}), there is a sequence $g_n\in C_c^\infty (\Rd \setminus B_{1/2})$ which converges to $g$ in $\sqrt{{\mathcal{E}}} + \|\cdot\|_{L^2}$  as $n\to \infty$. Hence, $w_n:=g_n/h\in C_c^\infty (\Rd \setminus B_{1/2})$ converges  to $f_2$ 
in the norm $\sqrt{\mathcal{Q}} +\|\cdot\|_{L^2(h^2(x)d x)}$ as $n\to \infty$. Finally, a sequence $f_1+w_n\in C_c^\infty (\Rd)$ converges in $\mathcal{D}(\mathcal{Q})$ to $f$ as $n\rightarrow \infty$. This means that $C^\infty_c(\Rd)$ is dense in $\mathcal{D}(\mathcal{Q})$ with respect to the norm $\sqrt{\mathcal{Q}} +\|\cdot\|_{L^2(h^2(x)d x)}$.

Using Lemma \ref{lem:den2} and similar arguments as for the density of $C^\infty(\Rd)\cap \mathcal{D}(\mathcal{Q})$  in $\mathcal{D}(\mathcal{Q})$ we see that $C^\infty_c(\Rd\setminus\{0\})$ is dense in $C^\infty_c(\Rd)$ in the norm $\sqrt{\mathcal{Q}} + \|\cdot\|_{L^2(h^2(x)dx)}$, hence also in $ \mathcal{D}(\mathcal{Q})$. Since $hf\in C^\infty_c(\Rd\setminus\{0\})$ for $f\in C^\infty_c(\Rd\setminus\{0\})$, we obtain that $C^\infty_c(\Rd\setminus\{0\})$ is dense in $\mathcal{D}(\overline{\mathcal{E}})$.
By Lemma \ref{lem:form_domain}, $\tilde{\mathcal{E}}$ and $\overline{\mathcal{E}}$ coincide on $C^\infty_c(\Rd\setminus\{0\})$. Since $\mathcal{D}(\tilde{\mathcal{E}})\subset \mathcal{D}(\overline{\mathcal{E}})$ and both spaces are complete, they are equal.
\end{proof}

\subsection{Davies' method}\label{sec:Dm}

Below we show that the upper bound for $\tp$ in the subcritical and critical cases follows from the near supermedian property of $H$ in Corollary~\ref{cor:sh}, Theorem~\ref{prop:Hilbert},
the Sobolev inequality of Frank, Lieb, Seiringer \cite[ Corollary 2.5]{MR2425175} for  the quadratic form of $\Delta^{\alpha/2}+\kappa^*|x|^{-\alpha}$ on $C^\infty_c(\Rd)$ and, indeed, Davies' method. 

\begin{proof}[Second proof of the upper bound in Theorem~\ref{thm:main}]
It suffices to verify that there is a constant $C= C(\alpha , d)$
such that for all $t>0$ and $a.e.$ $x,y \in \bbfR^d$,
\begin{equation}\label{est_ker_sharp}
   \tilde p(t,x, y) \leq C p(t,x, y)H (t,x) H(t,y).
\end{equation}
We first apply \cite[Theorem A]{MR2064932} with $\varphi_s(x)=H(s,x)$
to secure a constant $C=C(d,\alpha)>0$ such that, for all $t>0$ and a.e $x,y\in \Rd$,
\begin{equation}\label{est_ker}
   \tilde p(t,x, y) \leq C t^{-\frac{2d-\alpha}{\alpha}}H (t,x) H(t,y).
\end{equation}
Due to Theorem~\ref{prop:Hilbert} we only have to check assumptions (1)-(4) in \cite[Theorem A]{MR2064932}. Of these, (4) holds with $c_0=1$. Since  $\delta\leq (d-\alpha)/2$ we get that $H(s,\cdot),1/H(s,\cdot)\in L^2_{loc}(\Rd)$, hence (2) there holds. 
The assumption (3) is a consequence of Corollary \ref{cor:sh}.
Lemma \ref{lem:form_domain} implies that the quadratic form corresponding to $\tilde p$ is equal to
$$\tilde{\mathcal{E}}(f,f)=  \mathcal{E}(f,f) - \int_{\bbfR^d} f^2(x) q(x) d x,\quad f\in C^\infty_c(\Rd).$$
By \cite[Corollary 2.5]{MR2425175} we have 
$$\|f\|^2_{L^q}\leq c (\tilde{\mathcal{E}}(f,f)+\|f\|^2_{L^2}),\quad f\in C^\infty_c(\Rd),$$
for $q=2+\alpha/(d-\alpha)$ and $c=c(d,\alpha)$. By Theorem~\ref{prop:Hilbert},  $C^\infty_c(\R^d)$  is dense in $\mathcal{D}(\tilde{\mathcal{E}})$, so the above Sobolev inequality holds on the whole $\mathcal{D}(\tilde{\mathcal{E}})$. This implies (1) in \cite[Theorem A]{MR2064932} with $j=1+\alpha/(2(d-\alpha))$ there.  By \cite[Theorem A]{MR2064932}, \eqref{est_ker} holds true.
By \eqref{est_ker}, 
$$
\tp(1,x,y) \le c H(1,x)H(1,y), \qquad a.e. \mbox{ for } \,x,y\in\Rd.
$$
Since $p(1,x,y) \ge c/(1+ |x-y|^{d+\alpha})$, 
\begin{align}\label{eq:tp_est_compact10}
	\tp(1,x,y) \le C_R p(1,x,y) H(1,x) H(1,y), \qquad |x|,|y| <2R,
\end{align} 
for all $R>0$. In the remaining case by symmetry we only need to consider $|x| \le |y|$. Let $R = (2\kappa)^{1/\alpha}$ and $|y| \ge 2R$. Note that 
$p(1-s,z,y) \le c p(1,x,y)$ for all $s \in (0,1)$ and $z \in B(0,R)$. By \eqref{eq2:int_tpinfty}
and \eqref{eq:tp_q},
\begin{align*}
	\tp(1,x,y) &\le 2p(1,x,y) + 2\int_0^1 \int_{|z|<R}\tp(s,x,z) q(z) p(1-s,z,y)\,dz\,ds \\
	& \le 2p(1,x,y) + 2c\int_0^1 \int_{|z|<R}\tp(s,x,z) q(z) p(1,x,y)\,dz\,ds \\ 
	& \le C p(1,x,y)H(1,x) \le C p(1,x,y)H(1,x)H(1,y).
\end{align*}
By this and \eqref{eq:tp_est_compact10},
\begin{align*}
	\tp(1,x,y) \le C p(1,x,y)H(1,x)H(1,y) \quad \text{ for almost all } \; x,y \in \bbfR^d.
\end{align*}
By scaling we get \eqref{est_ker_sharp}.
\end{proof}

\end{document}